\documentclass{amsart}
\usepackage[T1]{fontenc}
\usepackage[latin9]{inputenc}
\usepackage{amsmath,amsthm,amsfonts,amscd,amssymb,eucal,latexsym,mathrsfs}
\usepackage{stmaryrd}
\usepackage{enumerate}
\usepackage{hyperref}

\usepackage{tikz}
\usetikzlibrary{positioning}
\usepackage{float}
\usepackage{MnSymbol}
\usetikzlibrary{matrix}

\theoremstyle{plain}
\newtheorem{theorem}{Theorem}
\newtheorem{corollary}[theorem]{Corollary}
\newtheorem{lemma}[theorem]{Lemma}
\newtheorem{proposition}[theorem]{Proposition}

\theoremstyle{definition}
\newtheorem{remark}[theorem]{Remark}

\newtheorem{example}[theorem]{Example}

\newtheorem{question}[theorem]{Question}

\newtheorem{definition}[theorem]{Definition}

\newcommand{\bb}[1]{\llbracket {#1}\rrbracket}
\newcommand{\acts}{\curvearrowright}
\newcommand{\G}{\Gamma}
\newcommand{\p}{\varphi}
\newcommand{\e}{\varepsilon}
\newcommand{\h}{\mathfrak{h}}

\newcommand{\s}{\mathbf{s}}
\renewcommand{\r}{\mathbf{r}}
\renewcommand{\d}{\mathrm{d}}
\newcommand{\IZ}{\mathbb{Z}}

\newcommand{\IQ}{\mathbb{Q}}
\newcommand{\M}{\mathrm{M}}

\DeclareMathOperator{\tr}{\mathrm{tr}}
\DeclareMathOperator{\ran}{\mathrm{ran}}
\DeclareMathOperator{\dom}{\mathrm{dom}}
\DeclareMathOperator{\var}{\mathrm{var}}
\DeclareMathOperator{\vspan}{\mathrm{span}}
\DeclareMathOperator{\bsigma}{\mathbf{\Sigma}}
\DeclareMathOperator{\HA}{\mathrm{HA}}
\DeclareMathOperator{\SA}{\mathrm{SA}}
\DeclareMathOperator{\NSA}{\mathrm{NSA}}
\DeclareMathOperator{\B}{\mathrm{Be}}

\DeclareMathOperator{\Sym}{\mathrm{Sym}}
\DeclareMathOperator{\Aut}{\mathrm{Aut}}
\DeclareMathOperator{\Fix}{\mathrm{Fix}}
\DeclareMathOperator{\ssi}{\Leftrightarrow}

\title{A free product formula for the sofic dimension}

\author{Robert Graham}
\author{Mikael Pichot}

\begin{document}
\begin{abstract}
It is proved that if $G=G_1*_{G_3}G_2$ is free product of probability measure preserving $s$-regular ergodic discrete groupoids amalgamated over an amenable subgroupoid $G_3$, then the sofic dimension $s(G)$ satisfies the equality
\[
s(G)=\h(G_1^0)s(G_1)+\h(G_2^0)s(G_2)-\h(G_3^0)s(G_3)
\]
where $\h$ is the normalized Haar measure on $G$.
\end{abstract}

\maketitle

Let $G$ be a group. The sofic dimension of $G$ is  an  asymptotic invariant that accounts for the number of unital maps \[
\sigma\colon F^n_{\pm}\to \Sym(d)
\] 
from the ``Cayley ball'' $F^n_{\pm}$ of radius $n$ in $G$ into the symmetric group $\Sym(d)$, where
 $F\subset G$ is a finite set, $n$ is an integer, $d$ is a `very large' integer and the  maps $\sigma$ are  multiplicative and free  up to an error $\delta>0$ relative to the normalized Hamming distance on $\Sym(d)$ (see \textsection\ref{S- s(G)} below). If $\SA(F,n,\delta,d)$ is the (finite) set of all such maps, and $\NSA:=\left |\{\sigma_{|F},\ \sigma\in \SA\}\right|$, the sofic dimension of $F$ is:
\[
s(F)=\inf_{n\in\mathbb{N}}\inf_{\delta>0}\limsup_{d\rightarrow\infty}\frac{\log \NSA(F,n,\delta,d)}{d\log d}
\]
(so the limit is on $d$ first, and then on $\delta$ and $n$). This definition  was considered in \cite{DKP1} and \cite{DKP2}. It is a combinatorial version of Voiculescu's (microstate) free entropy dimension $\delta(F)$,  which can be defined by a similar formula involving maps
\[
\sigma\colon F^n_{\pm}\to U(d)
\] 
into the unitary group $U(d)$ (see \cite{Voi96, Jun02}). It can be shown that the value of $s(F)$ doesn't dependent on the finite generating set $F$ of $G$ and is therefore denoted $s(G)$. A limiting process allows to define of $s(G)$ for an arbitrary group $G$.

The definition of the sofic dimension can be extended to probability measure preserving (pmp) actions of countable groups, their orbit equivalence relations, and more generally to discrete pmp groupoids. We refer to \cite[Definition 2.3]{DKP2} for the general groupoid definition. An interesting feature of $s$ is to provide combinatorial proofs of statements in orbit equivalence theory (for example, Corollary 7.5 in  \cite{DKP1} reproves Gaboriau's theorem that the free groups on $p$ generators are pairwise non orbit equivalent using the counting method).

Let $G$ be a pmp groupoid and assume that $G=G_1*_{G_3} G_2$ is an amalgamated free product over a subgroupoid $G_3$.  A free product formula of the form 
\[
s(G)=s(G_1)+s(G_2)-s(G_3)
\] 
is known to hold in the following cases (under some technical assumptions, for example ``finitely generated'' and/or ``$s$-regularity''):
\begin{enumerate}
\item $G$, $G_1$, $G_2$, and  $G_3$ are a pmp equivalence relations on $(X,\mu)$ and $G_3$ is amenable as an equivalence relation: see \cite[Theorem 1.2]{DKP1}.
\item $G$, $G_1$, $G_2$, and  $G_3$ are countable groups and $G_3$ is an amenable group: see \cite[Theorem 4.10]{DKP2}.
\item $G$ is the crossed product groupoid $G:=G_1*G_2\ltimes X$ of a pmp action $(G_1* G_2)\acts (X,\mu)$, where $G_1$ and $G_2$ are countable groups and $X$ is a standard probability space. Here $G_3$ is assumed to be the trivial group but the action of $G_1*G_2$ is \emph{not necessarily free} (if free this is covered by (1)): see \cite[Theorem 6.4]{DKP2}.
\end{enumerate} 
The general strategy to establish this sort of formula was devised by Voiculescu for the free entropy dimension: see in particular \cite{Voi91,Voi96,Voi98}. 

The proofs of the above results in \cite{DKP1, DKP2}  apply distinct tools to handle the amenable amalgamated part, namely the Connes--Feldmann-Weiss theorem in (1), and the Ornstein--Weiss quasi-tiling theorem in (2). This was a reason why it was hardly conceivable to incorporate an amenable amalgamated subgroup $G_3$ in (3); in fact,  the technical details would presumably  (to quote  \cite[\textsection 6]{DKP2}) be `formidable' even if the action of $G_3\acts X$ is essentially free.

We follow a different approach here, based on the use of Bernoulli shifts as a ``correspondence principle'' 
\begin{center}
groupoids $\leftrightsquigarrow$ equivalence relations 
\end{center}
by which we mean that proving a result for (pmp) equivalence relations  `automatically' implies an a priori more general statement for (pmp) groupoids in a variety of situation, and in particular for the computation of $s$ (see \textsection \ref{S- applications} for more details).

The exact assumptions that we need for the free product formula are described in the following statement, which is the main result of this paper: 

\begin{theorem}\label{T - FPF}
Let $G$ be a discrete pmp groupoid of the form $G=G_1*_{G_3}G_2$, where $G_1,G_2$ are $s$-regular ergodic subgroupoids of $G$ and $G_3$ is an amenable groupoid, then 
\[
s(G)=\h(G_1^0)s(G_1)+\h(G_2^0)s(G_2)-\h(G_3^0)s(G_3)
\]
where $\h$ is the normalized Haar measure on $G$ and $G^0$ is the object space of $G$.
\end{theorem}

The more technical assumptions in this result can be weakened slightly. For example, one way to remove the $s$-regularity assumption, following Voiculescu's idea, see e.g.\ \cite[Remark 4.8]{Voi98}, is to replace the $\limsup$ in the definition of $s(F)$ by a limit along a fixed ultrafilter $\omega$. What is rather unclear  is the extend to which the assumption that $G_3$ is amenable is essential. 
Cohomological tools can be used to prove a similar formula for the first $L^2$ Betti number $\beta_1$ under the much weaker assumption that $\beta_1(G_3)=0$. (This  is a result of L\"uck see \cite[Theorem A.1]{BDJ} for the group case.) Furthermore Mineyev and Shlyakhtenko \cite{MS}  have shown that Voiculescu's `non-microstate' free entropy dimension $\delta^*$ satisfies $\delta^*(G)=\beta_1(G)-\beta_0(G)+1$ for any finitely generated group $G$, and therefore we have the formula
\[
\delta^*(G_1*_{G_3}G_2)=\delta^*(G_1)+\delta^*(G_2)-\delta^*(G_3)
\]
where $G_1$ and $G_2$ are finitely generated groups and $G_3$ is a group such that $\beta_1(G_3)=0$. A fundamental relation between the microstate and the non-microstate approach to free entropy is provided by the Biane--Capitaine--Guionnet inequality $\delta\leq \delta^*$ \cite{BCG}. A free product formula for $\delta_0$ has been established in \cite{BDJ} for amalgamation of ($\delta_0$-regular) groups over an amenable subgroup (where $\delta_0\leq \delta$ is a technical modification of $\delta$ not depending on the generating set of the group, see  \cite[Section 6]{Voi96} and \cite{Voi98}). We also note that the above correspondence principle for $\delta_0$  is probably less useful as the amenable part can always be handled uniformly using the hyperfiniteness of von Neumann algebra $LG_3$ (see in particular \cite{Jun}; for example, the proof in \cite{BDJ} in the group case doesn't rely on quasi-tilings). Concerning pmp equivalence relations,  a free product formula has been established by Gaboriau \cite{Gab} for the cost,  allowing for amalgamations over amenable subrelations, and by Shlyakhtenko \cite{Shl} for $\delta_0$, for free product with trivial amalgamation.  

\begin{question}
Can the assumption that $G_3$ is amenable in Theorem \ref{T - FPF} be weakened (for example to $\beta_1(G_3)=1$)?
\end{question}

The paper is organized as follows:  \textsection\ref{S- s(G)} and \textsection\ref{S - groupoid actions} establish basic facts about pmp groupoids and their actions. In the case of $s(G)$, the correspondence ``groupoids $\leftrightsquigarrow$ equivalence relations'' is achieved by  using the formula  $s(G)=s(G\ltimes X_0^G)$, where $G\acts X_0^G$ is a Bernoulli shift, see Theorem  \ref{T-action Bernoulli infinite} in \textsection \ref{S - Bernoulli corresp} (other applications of the correspondence principle are given in \textsection\ref{S- applications}). The proof of this formula uses the idea  in a result of L. Bowen \cite[Theorem 8.1]{Bowen2010} for the sofic entropy, as explained in \textsection\ref{S - Bowen groupoids}. Other difficulties inherent to the groupoid setting are dealt with in \textsection \ref{S - overlapping gen}, \textsection\ref{S - linear}, \textsection\ref{S -HA} (these difficulties were avoided in \cite{DKP2} by working with groups and their actions  rather than with general groupoids).  
The proof of Theorem \ref{T-action Bernoulli infinite} is given in Section \ref{S - Bernoulli corresp}. In \textsection \ref{S -Scaling} we prove a scaling formula for $s(G)$. In \textsection\ref{S- applications} we prove Theorem \ref{T - FPF} by putting together these ingredients.

\section{Review of $s(G)$}\label{S- s(G)}

Recall that a discrete standard Borel groupoid $G$ with base (=set of objects) $G^0$, source map $\s\colon G\to G^0$ and range map $\r\colon G\to G^0$, is said to be \emph{probability measure preserving (pmp)} with respect to a Borel probability measure $\mu$ on $G^0$ if the left and right Haar measures  $\h$ and $\h^{-1}$  on $G$ coincide:
\[
\h=\h^{-1}
\]
where
\[
\h(A):= \int_{G^0} |A^e| \, \d\mu (e)\text{  and  }  
\h^{-1}(A):= \int_{G^0} |A_e| \, \d\mu (e) 
\]
and $A\subset G$ is a Borel set with $A^e:=\r^{-1} (e)\cap A$ and $A_e:=\s^{-1} (e)\cap A$ for  $e\in G^0$. Then $\h_{|G^0}=\h^{-1}_{|G^0}=\mu$ so we simply denote by $\h$ the measure $\mu$ on $G^0$.

A \emph{bisection} is a Borel subset $s\subset G$ such that the restrictions of $\s$ and $\r$ to $s$ are Borel isomorphisms onto $G^0$. The set of bisections form a group called the \emph{full group} of $G$ and are denoted $[G]$. A \emph{partial bisection} is a Borel subset $s\subset G$ such that $\s$ and $\r$ are injective in restriction to $s$. The set of partial bisection form a Polish inverse monoid called the \emph{full inverse semigroup} (or the full pseudogroup) of $G$ and denoted $\bb G$. For $s\in \bb G$ let $\dom(s):=s^{-1}s\subset G^0$ and  $\ran(s):=ss^{-1}\subset G^0$.

For example, if $G:=\{1,\ldots, d\}^2$ is the transitive equivalence relation on the set $\{1,\ldots d\}$ with $d\in \IZ_{\geq 1}$ elements, then $[G]=\Sym(d)$ is the symmetric group on $d$ letters and $\bb G$ is the inverse semigroup of partial permutations. We denote the latter by $\bb d$.

The semigroup $\bb G$ (and $[G]\subset \bb G$) is Polish with respect to  the \emph{uniform distance} 
\[
|s-t|:=\h\{e\in G^0\mid se\neq te\}.
\]
If $G=\{1,\ldots,d\}^2$ then the uniform distance is the normalized Hamming distance on $\bb d$.

The (von Neumann) \emph{trace} on $\bb G$ is given by
\[
\tau(s):=\h(s\cap G^e)=\h\{e\in G^0\mid se=e\}.
\]
It is the restriction to $\bb G\subset LG$ of the finite trace on the von Neumann algebra $LG$ of $G$. 

We have
\[
|s-t|=\tau(s^{-1}s)+\tau(t^{-1}t)-\tau(s^{-1}st^{-1}t)-\tau(st^{-1}).
\]

We will write $\tr$ for the trace on $\bb d$. So 

\begin{center}
$\tr(\sigma)=\frac 1 d\times$number of fixed points of $\sigma \in \bb d$.
\end{center}

If $F\subset \bb G$ is a finite subset and $n\in \IZ_{\geq 1}$ then $F_\pm^n$ denotes the set of all products of at most $n$ elements of $F_\pm:= F\cup F^{-1}\cup\{\mathrm{Id}\}$. For $F\subset \bb G$ we let $\bsigma F\subset \bb G$ denote the  set of sums of elements of $F$ with pairwise orthogonal domains and pairwise orthogonal ranges. 

By definition,  $G$ is \emph{sofic} if its full inverse semigroup $\bb G$ is sofic: 

\begin{definition}\label{D - sofic}
A pmp groupoid $G$ is called sofic  if for every finite set $F\subset \bb G$, $\delta>0$ and $n\in \IZ_{\geq 1}$ there exist $d\in \IZ_{\geq 1}$ and a map 
\[
\sigma\colon \bsigma F_\pm ^n\to  \bb d
\] 
such that
\[
|\sigma(st)-\sigma(s)\sigma(t)|<\delta \tag{i}
\] 
for every $s,t\in  \bsigma F_\pm ^n$ such that $st\in \bsigma F_\pm ^n$ ($\sigma$ is \emph{$\delta$-multiplicative}) and
\[
 |\tr\circ \sigma(s)-\tau(s)|<\delta \tag{ii}
 \]
for every $s\in  F_\pm ^n$ ($\sigma$ is \emph{$\delta$-trace-preserving}).
\end{definition}    
  
 \begin{remark}
If $G$ is a group, one can replace $\bb d$ by $\Sym(d)$ as is easily seen.
 \end{remark} 
  
\begin{remark}
The notion of sofic pmp equivalence relations was introduced by Elek and Lippner in \cite{EL} in terms of graph approximation and in \cite{Oz} by requiring that  $\bb G$ is sofic as in the definition above (compare \cite{DKP1,DKP2}). The sofic property was first considered for groups by Gromov and Weiss  in terms of (Cayley) graph approximation and was studied by Elek and Szabo.   It is a simultaneous generalization of amenability and the LEF property of Vershik and Gordon (see \cite{Pe} for more details).
\end{remark}  
  
 Let 
 \[
 \SA(F,n,\delta,d):=\{\sigma \colon \bsigma F_\pm ^n\to \bb d\text{ satisfying  (i), (ii)} \}
 \] 
and define
 \[
 |\SA(F,n,\delta,d)|_E:=|\{\sigma_{|E}\mid \sigma\in \SA(F,n,\delta,d)\}|
  \]
  for $E\subset F$.
 
 \begin{definition} For $E\subset F\subset \bb G$ finite, $n\in\IZ_{\geq 1}$ and $\delta > 0$ define successively
\begin{align*}
s_E (F,n,\delta ) &:= \limsup_{d\to\infty} \frac{\log |\SA (F,n ,\delta ,d)|_E}{d \log d}  \\
s_E (F,n) &:= \inf_{\delta > 0} s_E (F,n,\delta )\\
s_E (F) &:= \inf_{n\in\IZ_{\geq 1}} s_E (F,n) 
\end{align*}
If $K\subset \bb G$ is an arbitrary subset the sofic dimension of $K$ is 
\[
s(K):=\sup_E\inf_{F} s_E(F)
\]
where $E\subset F\subset K$ are finite subsets. The sofic dimension of $G$ is $s(G):=s(\bb G)$. One defines similarly the lower sofic dimension $\underline s$ and the $\omega$ sofic dimension $s^\omega$ for a ultrafilter $\omega$ on $\IZ_{\geq 1}$ by replacing $\limsup_{d\to \infty}$ by $\liminf_{d\to \infty}$ and $\lim_{d\to \omega}$ respectively.
 \end{definition}
 
 Voiculescu's regularity condition reads:
 
\begin{definition} 
    A pmp groupoid $G$ is \emph{$s$-regular} if $\underline s(G)=s(G)$. 
\end{definition}  
  
Finally we recall

\begin{definition}
A subset $K\subset \bb G$ is \emph{transversally generating} if for any $t\in \bb G$ and $\e>0$ there exist $n\in \IZ_{\geq 1}$ and  $s\in  K_\pm^n$ such that $|t-s|\leq \e$.
\end{definition}  
 
This definition appears in \cite[Definition 2.4]{DKP1} where it is called ``dynamically generating''.  The more classical notion of generating set for pmp equivalence relations (and groupoids) (as in \cite[Definition 2.2]{DKP1}) is that of Connes--Feldmann--Weiss. While being distinct notions, a groupoid is finitely generated in the Connes--Feldmann--Weiss sense if and only if it is transversally finitely generated (by an argument similar  to that in \cite[Proposition 2.6]{DKP1}) so `finitely generated'  is unambiguous for groupoids (and coincide with the usual notion in the group case).

The following result is proved in \cite[Theorem 4.1]{DKP1}.

\begin{theorem}[Invariant of $s$ under orbit equivalence]\label{T - OE invariance}
Let $R$ be pmp equivalence relation and $K,L$ be transversally generating sets. Then $s(K)=s(L)$, $\underline s(K)=\underline s(L)$ and $s^\omega(K)=s^\omega(L)$. 
\end{theorem}

The result in \cite{DKP1} is stated for finitely generated equivalence relations but the same proof works in the general case (as does the proof of \cite[Theorem 1.2]{DKP1}). The proof in full generality for groupoids is given in \cite[Theorem 2.11]{DKP2}. We won't use this more general result here but will rather deduce it from Theorem \ref{T - OE invariance} as an illustration of the correspondence principle.

\begin{remark}\label{compare norms} It is sometimes convenient to use the 2-norm on $LG$ and its restriction to $\bb G$:
\[
\|s-t\|_2^2:=\tau((s-t)(s-t)^{-1})=\tau((s-t)^{-1}(s-t))
\]
Observe
\[
\|s-t\|_2^2\geq |s-t|
\] 
(with equality on $[G]\subset \bb G$) as
\[
\|s-t\|_2^2=\tau(s^{-1}s)+\tau(t^{-1}t)-2\tau(st^{-1})
\]
and $\tau(st^{-1})\leq \tau(s^{-1}st^{-1}t)$. 
\end{remark}

\section{Actions of groupoids}\label{S - groupoid actions}

 Let $X$ be a standard Borel space endowed with a Borel fibration $p \colon X\to G^0$ where $G^0$ is the base of $G$. If $(\mu^e)_{e\in G^0}$ is a Borel field of probability measures on $X$ we define  a probability measure $\mu$ on $X$ by $\mu:=\int_{G^0} \mu^e\mathrm{d}\h(e)$ where $\h$ is the invariant Haar measure on $G^0$. Recall that a pmp  action of $G$ on the fibered space $(X,\mu)$ is a measurable map 
\[
G*_{G^0}X\ni (g,x)\mapsto gx\in X
\]
(where $G$ fibers via the source map $\s: G\to X_G$) satisfying the usual axioms of an action, and such that 
\[
g_*\mu^{\s(g)}=\mu^{\r(g)}
\]
for ae $g\in G$.  Groupoid actions are denoted  $G\acts X$. The crossed product groupoid is the fiber bundle $G\times_{G^0} X$ endowed with groupoid law defined by $(s,x)(t,y)=(st,y)$ whenever $t(y)=x$.

\begin{example}[Bernoulli shifts] Given a pmp groupoid $G$ with invariant Haar measure $\h$ and a probability space $(X_0,\mu_0)$, consider the probability space 
\[
(X_0^G,\mu):=\int_{G^0} (X_0^{G^e}, \mu_0^{\bigotimes{G^e}})\d\h(e),
\] 
(where $X_0^{G^e}:=\prod_{G^e}X_0$ is the infinite Cartesian  product over $G^e:=\r^{-1}(e)$), endowed with the fibration $X_0^G\to G^0$ and field of measures $(\mu_0^{\bigotimes{G^e}})_{e\in G^0}$. Every element $x\in X_0^G$ can be viewed as a sequence $x:=(x_t)_{t\in G^{e}}$ of elements of $X_0$. The Bernoulli action  $G\acts X_0^G$ is given by 
\[
s ((x_t)_{t\in G^{\r(s)}}):=((x_{s^{-1}t})_{t\in G^{\r(s)}}).
\]  
\end{example}

\begin{remark}
The notion of groupoid action has long been used in ergodic group theory (see for example  \cite{Ram}). They can equivalently be described as actions on bundles (cocycles) as above, or as groupoid extensions which are fiber bijective. In \cite{Bow2} L. Bowen discusses Bernoulli shifts using the latter description.   
\end{remark}

We first prove a few lemmas that will be used in the proof of Theorem \ref{T - FPF}.

\begin{lemma}\label{L - subshifts are Bernoulli}
Let $G\acts X_0^G$ be a Bernoulli action and $H\subset G$ be an ergodic subgroupoid with $G^0=H^0$. 
Then the action $H\acts X_0^{G}$ is isomorphic to a Bernoulli shift over $H$.
\end{lemma}
\begin{proof}
By the von Neumann selection theorem, we can find a measurable section $s\colon RG\to G$ of $(\s,\r)\colon G\rightrightarrows RG$ (=the pmp equivalence relation associated with $G$) such that that $s(RG^0)=G^0$ and $s(RH)\subset H$ and, since the set $G(e)/H(e)$ is countable for $e\in G^0$, measurable sections $(g_j \colon D_j\subset G^0\to SG)_{j\in J}$ of $SG\to G^0$ such that for ae $e\in G^0$, $G(e)=\bigsqcup_{j\in J} H(e)g_{j}(e)$. By ergodicity we may assume $\h(D_j)=1$.
Let $(\p_i)_{i\in I}$ be a sequence in $\Aut(G^0)$ such that $\{RH[\p_ie]\}_{i\in I}$ form a partition of $RG[e]$ for  ae $e\in G^0$ (see \cite{FSZ}). 
 Then 
\[
G^e=\bigsqcup_{i\in I,\, j\in J} \bigsqcup_{(e,f)\in RH}H(e)g_j(e)s(e,\p_i^{-1}f)\\
\]
since any $g\in G^e$ can be written uniquely in the form $h_0g_j(e)s(e,\p_i^{-1}f)$ for $i$ and $f$ such that  $(e,f)\in RH$ and $f=\p_i\s(g)$ so $gs(e,\p_i^{-1}f)^{-1}\in G(e)$ and $h_0=gs(e,\p_i^{-1}f)^{-1}g_j(e)^{-1}\in H(e)$.

Consider the measurable field of maps $\psi_e:X_{0}^{G^{e}}\rightarrow(X_{0}^{I\times J})^{H^{e}}$
defined by sending $x\in X_{0}^{G^{e}}$ to 
\[
\left((x_{h_0g_j(e)s(e,\p_i^{-1}f)})_{(i,j)\in I\times J}\right)_{h_0\in H(e), (e,f)\in RH}.
\]
These maps are measure preserving and if we consider the Bernoulli action
of $H$ with base $X_{0}^{I\times J}$ then we see it is $H$-equivariant:
for $h\in H^e$ say $h=h_1^{-1}s(d,e)^{-1}$ where $h_1\in H(e)$
\begin{align*}
  \psi_d(h(x))&=\psi(h(x_{h_0g_j(e)s(e,\p_i^{-1}f)})_{(i,j)\in I\times J,h_0\in H(e),(e,f)\in RH})\\&=\psi((x_{s(d,e)h_1h_0g_j(e)s(e,\p_i^{-1}f)})_{(i,j)\in I\times J,h_0\in H(e),(e,f)\in RH})\\
 & =\left((x_{s(d,e)h_1h_0g_j(e)s(e,\p_i^{-1}f)})_{(i,j)\in I\times J}\right)_{h_0\in H(e),(e,f)\in RH}\\
 &=h\left((x_{h_0g_j(e)s(e,\p_i^{-1}f)})_{(i,j)\in I\times J}\right)_{h_0\in H(e),(e,f)\in RH}\\
 &=h(\psi_e(x))
\end{align*}
\end{proof}

We say a groupoid action $G\acts X$ is \emph{essentially free}
 if for ae $s\in G\setminus G^0$ 
\[
\mu^{\s(s)}(\Fix(s))=0
\]
where 
\[
\Fix(s)=\{x\in X^{\s(s)}\mid sx=x\}.
\]

\begin{lemma}\label{L - free groupoid actions are equivalence relations}
If the pmp groupoid action $G\acts X$ is essentially free then $G\ltimes X$
is a pmp equivalence relation.
\end{lemma}

\begin{proof}
Since $G\acts X$ is essentially free and $G_e$ is countable the set 
\[
X_0^e:=\{x\in X^e\mid sx\neq x,\ \forall s\in G_e,\ s\neq e\}
\] 
has measure 1  in $X^e$ for every $e\in A\subset G^0$ a measurable subset with $\h(A)=1$. Let $(X_0^A,\mu)\to (A,\h)$ be the measure fibration corresponding to $(X_0^e)_{e\in A}$.  Since $G$ is isomorphic to $G_{|A}\ltimes X_0^A$, we may assume that $sx\neq x$ for all $s\in G\setminus G^0$, $x\in X^{\s(s)}$.
Then for $(s,x)\in G\ltimes X$
\[
\r(s,x)=\s(s,x)\ssi x\in \Fix(s)
\]
so $\r(s,x)\neq\s(s,x)$ for every $s\in G\setminus G^0$ and $x\in X$. This shows that $G\ltimes X$ is an equivalence relation. It is an easy exercise to check that it is pmp (more generally if $G$ is pmp and $G\acts X$ is a pmp action, then $G\ltimes X$ is a pmp groupoid).  
\end{proof}

\begin{lemma}\label{lemma B0 def}
If $G$ is transversally finitely generated then so is $G\ltimes X$ for any ergodic pmp  action  $G\acts X$. 
\end{lemma}

\begin{proof}
Let $R$ be the orbit equivalence relation of $G\acts X$. Since $R$ is ergodic, there exists an ergodic automorphism $\theta\in \bb R$, which is orbit equivalent to a Bernoulli shift $\IZ\acts \{0,1\}^\IZ$ (\cite{Dye}). For $i\in\{0,1\}$ let $B_i$ be the cylinder set $\{x\in\{0,1\}^\mathbb{Z}\mid x_0=i\}$. Let $p_1,p_2$ be the projection in $L^\infty(X)$ corresponding to $B_0$ and $B_1$ in $\IZ\acts \{0,1\}^\IZ$.
If $F\subset \bb G$ be a finite transversally generating set for $G$, then $F\cup \{\theta,p_1,p_2\}$ is a finite generating set for $G\ltimes X$.
\end{proof}

\begin{lemma}
The Bernoulli action $G\acts X_0^G$ is essentially free
if $G$ has infinite fibers  (i.e.\ $|G^e|=\infty$ for ae $e\in G^{0}$), and the support of $\mu_0$ contains at least two points. If in addition $\mu_0$ is diffuse then the action is essentially free.
\end{lemma}

\begin{proof}
Let $s\in G^e$, $s\neq e$, such that $|G^{e}|=\infty$. We show that $\mu^{e}(\Fix(s))=0$.
If $\mathbf{r}(s)\neq \mathbf{s}(s)$ this is clear so we assume $s\in G(e)$. Note that $(x_t)_{t\in G^e}$ is $\langle s\rangle$-invariant if and only if $x_{s^{n}t}=x_t$ for all $n\in \IZ$, $t\in G^e$. Thus we can find an infinite family of pairwise disjoint pairs $\{s_i,t_i\}_{i\in I}$, $s_i\neq t_i\in G^e$, such that $x_{s_i}=x_{t_i}$ for every $x\in \Fix(s)$ and $i\in I$.
Since $|I|=\infty$, the set of $x\in X_0^{G^e}$ such that $x_{s_i}=x_{t_i}$ is negligible, so $\mu^e(\Fix(s))=0$. If in addition $\mu_0$ is diffuse then the set of $x\in X_0^{G^e}$ such that $x_{s}=x_{e}$ is negligible  so $\mu^{e}(\Fix(s))=0$ in this case.
\end{proof}

\begin{lemma}
If G is ergodic with infinite fibers then the Bernoulli action $G\acts X_0^G$ is ergodic for any base space $(X_0,\mu_0)$. 
\end{lemma}

\begin{proof}
If $A\subset X_0^G$ is a nonzero $G$-invariant subset and $p\colon X_0^G\to G^0$, then $p(A)$ is $G$-invariant and nonzero therefore $\delta:=\mu^e(A^e)\neq 0$ and is almost surely constant by the ergodicity of $G$. Let $\e>0$ be arbitrary. Since $|G^e|=\infty$ almost surely we can find by the von Neumann selection theorem a nonzero measure field $(F^e)_{e\in B}$ of finite subsets of $(G^e)_{e\in B}$ such that $\mu^e(A^e\cap A_{|F^e})>\delta -\e$ and a section $s$ of $\s$ such that  $s_eF^{e}\cap F^{\r(s_e)}=\emptyset$ for ae $e\in B$. By invariance $\mu^{\r(e)}(A^{\r(s_e)}\cap A_{|s_eF^e})>\delta-\e$ so
\[
\mu^{\r(e)}(A_{|F^{\r(s_e)}}\cap A_{|s_eF^e})>\delta-2\e
\]
and $\mu^{\r(e)}(A_{|F^{\r(s_e)}}\cap A_{|s_eF^e})=\mu^{\r(e)}(A_{|F^{\r(s_e)}})\mu^{\r(e)}(A_{|s_eF^e})$ also 
\[
\mu^{\r(e)}(A_{|F^{\r(s_e)}}\cap A_{|s_eF^e})<\delta^2+6\e\delta + 2 \e^2.
\]
Letting $\e\to0$, we get $\delta^2\geq \delta$ so $\delta\geq 1$ and $\mu(A)=1$.
 \end{proof}

\begin{lemma}\label{L - free product decompositions actions}
Let $G\acts X$ be a groupoid action and suppose that $G=G_{1}*_{G_{3}}G_{2}$. Then 
\[
G\ltimes X\simeq(G_{1}\ltimes X_{|1})*_{(G_{3}\ltimes X_{|3})}(G_{2}\ltimes X_{|2})
\]
where $X_{|i}=p^{-1}(G_{i}^{0})$. 
\end{lemma}

\begin{proof}
First note that $G_{i}$ acts on $X_{i}$ and
that $G_{i}\ltimes X_{|i}$ is naturally a subgroupoid of $G\ltimes X$.
Given an arbitrary groupoid $H$ and groupoid morphisms $f_{1}:G_{1}\ltimes X_{|1}\rightarrow H$,
$f_{2}:G_{2}\ltimes X_{|2}\rightarrow H$ with $f_{1}\big|_{G_{3}\ltimes X_{3}}=f_{2}\big|_{G_{3}\ltimes X_{|3}}$
we want to show there is a unique morphism $k\colon G\ltimes X\rightarrow H$
such that the following diagram commutes (where the unlabeled edges
are the inclusion map).
\begin{center}
\begin{tikzpicture}
    \matrix (m) [
      matrix of math nodes,
      row sep=1.5cm,
      column sep=.7cm,
      text height=1.5ex,
      text depth=0.25ex
    ] {
            & {G_3}\ltimes X_{|3} &           \\
      G_{1}\ltimes X_{|1} &  G\ltimes X    & G_{2}\ltimes X_{|2} \\
			& H & \\			
    };
    \path[->]        (m-1-2) edge node[left] {}(m-2-1)
                     (m-1-2) edge node[right] {} (m-2-3)
                     (m-2-1) edge node[left] {}  (m-2-2)
					(m-2-3) edge node[left] {}  (m-2-2)
					(m-2-1) edge node[left] {$f_1$}  (m-3-2)
					(m-2-3) edge node[right] {$f_2$}  (m-3-2)
					(m-2-2) edge node[left] {$k$}  (m-3-2);
\end{tikzpicture}
\end{center}
Since $k\big|_{G_{1}\ltimes X_{|1}}=f_{1}\big|_{G_{1}\ltimes X_{|1}}$, $k\big|_{G_{2}\ltimes X_{|2}}=f_{2}\big|_{G_{2}\ltimes X_{|2}}$
and the values of $k$ are determined on $G_{1}\ltimes X_{|1}$
and $G_{2}\ltimes X_{|2}$ which generate $G\ltimes X$, this
gives uniqueness. 
To
show it is well defined note that if $g_{i}\in G_{1}\ltimes X^{G_{1}}\cup G_{2}\ltimes X^{G_{2}}$
and $g_{1}...g_{n}=\mathrm{Id}_{\mathbf{s}(g_{1})}$ and if $g_{i}'$ are the
corresponding elements of $G_{1}$ and $G_{2}$ then $g_{1}'...g_{n}'=e$,
so using the fact that $G_1$ and $G_2$ are in free product over $G_3$  this shows that $k(g_{1})...k(g_{n})=\mathrm{Id}_{\mathbf{s}(f_{1}(g_{1}))}$
or $\mathrm{Id}_{\mathbf{s}(f_{2}(g_{1}))}$ as appropriate.
\end{proof}

\section{Overlapping generators}\label{S - overlapping gen}

Let $F\subset \bb G$ be a finite subset, $\pi$ be a partition of $F$, and $\sigma:F\to \bb d$ be a map. 
Denote $F_{|\pi}$ the set of $e\in G^0$ such that
\[
e \leq \prod_{s\in F} ss^{-1}
\]
and
\[
\forall s,t\in F,\ es=et\iff \pi(s)=\pi(t).
\]
Here and below we will view a partition  $\pi$ of $F$ as a map from $F$ to $\{1,...,\alpha\}=:\ran \pi$ for some $\alpha\in\mathbb{Z}_{\geq0}$.

Similarly let $F_{|\pi}^\sigma$ be the set of $e\in \{1,\ldots d\}$ (identified with the base space of the transitive relation on $\{1,\ldots, d\}$) such that 
\[
e \leq \prod_{s\in F} \sigma(s)\sigma(s)^{-1}
\]
and
\[
\forall s,t\in F,\ e\sigma(s)=e\sigma(t)\iff \pi(s)=\pi(t).
\]

The following lemma will be useful in the proof of Theorem \ref{T-action Bernoulli infinite}. We denote a projection onto a set $A$ by $p_{A}$.

\begin{lemma}\label{L - First Lemma} Given $F\subset \bb G$,  $n\geq 1$, let 
\[
F_n:=F\cup \{p_{{F_0}_{|\pi}}\mid F_0\subset  F_{\pm}^n, \, \pi\text{ partition of }F_0\}
\]
If $\sigma\in \SA(F_n,4n|F_{\pm}^n|+1,\delta,d)$, $F_{0}\subset F_\pm^n$, and $\pi$ is a partition of $F_0$  then
\[
\left|\mathfrak{h}(F_{0|\pi})-\frac{|F^{\sigma}_{0|\pi}|}{d}\right|<c_1(F,n)\delta.
\]
where $c_1(F,n)={176(3^{2n})|F_{\pm}|^{2n}}$.
\end{lemma}

\begin{proof}
Fix some $F_{0}\subset F_{\pm}^n$ and let $\pi_{1},\pi_{2},\ldots, \pi_{m}$ be the family of partitions of $F_{0}$ (so $m=\B(|F_0|)$ is the Bell number of $|F_0|$). For convenience let $p_j:=p_{F_{0|\pi_j}}$ and $p_j^\sigma:=p_{F^{\sigma}_{0|\pi_j}}$.

If $\pi_j(s)=\pi_j(t)$, $p_j= st^{-1}p_j$ so
\begin{align*}
&\left| \sigma(s)\sigma(t)^{-1}\sigma(p_j)-\sigma(p_j)\right|  =\left| \sigma(s)\sigma(t)^{-1}\sigma(p_j)-\sigma(st^{-1}p_j)\right|\\
 &\hspace{1cm} \leq\left| \sigma(s)\sigma(t)^{-1}\sigma(p_j)-\sigma(s)\sigma(t^{-1})\sigma(p_j)\right|+\left| \sigma(s)\sigma(t^{-1})\sigma(p_j)-\sigma(st^{-1}p_j)\right|\\
 &\hspace{1.5cm} \le4\delta+3\delta=7\delta
\end{align*}
by \cite[Lemma 3.1(8)]{DKP1}. 

Similarly if $\pi_j(s)\neq \pi_j(t)$ then
\begin{align*}
\left| \sigma(s)\sigma(t)^{-1}\sigma(p_j)-\sigma(st^{-1}p_j)\right|\leq & 7\delta.
\end{align*}
As $\tau(st^{-1}p_j)=0$ we have 
\begin{align*}
&\left|\tr(\sigma(s)\sigma(t)^{-1}\sigma(p_j))\right|  =\left|\tr(\sigma(s)\sigma(t)^{-1}\sigma(p_j))-\tau(st^{-1}p_j)\right|\\
  &\hspace{1cm} \leq\left|\tr(\sigma(s)\sigma(t)^{-1}\sigma(p_j))-\tr\circ\sigma(st^{-1}p_j)\right|+\left|\tr\circ\sigma(st^{-1}p_j)-\tau(st^{-1}p_j)\right|\\
  &\hspace{1.5cm} \leq7\delta+\delta=8\delta
\end{align*}
by \cite[Lemma 3.1 (4)]{DKP1}

Moreover
\begin{align*}
\left| \sigma(p_j)-\sigma(p_j)\sigma(p_j)\right| & =\left| \sigma(p_j^{2})-\sigma(p_j)\sigma(p_j)\right|<\delta.
\end{align*}
and since $\displaystyle \sum_{j=1}^mp_j=\prod_{s\in F_0}ss^{-1}$ using \cite[Lemma 3.5]{DKP1}
\begin{align*}
\left| \sum_{j=1}^{m}\sigma(p_j)-{\prod_{s\in F_{0}}}\sigma(s)\sigma(s)^{-1}\right| & <\left| \sigma\left ( \prod_{s\in F_{0}}ss^{-1}\right)-{\prod_{s\in F_{0}}}\sigma(s)\sigma(s)^{-1}\right| +150(2|F|+1)^{2n}\delta\\
 & \leq|F_{0}|4\delta+|F_{0}|2\delta+\delta+150(2|F|+1)^{2n}\delta\\
 &\leq157 (3^{2n})|F_{\pm}|^{2n}\delta.
\end{align*}
If $V$ denotes the set of integers $1\leq k\leq d$ such that for all $1\leq j\leq m$ and $s,t\in F_0$: 
\[
\sigma(s)\sigma(t)^{-1}\sigma(p_j)k=\sigma(p_j)k
\]
\[
\sigma(s)\sigma(t)^{-1}\sigma(p_j)k\neq k
\]
\[
\sigma(p_j)k=\sigma(p_j)\sigma(p_j)k
\]
\[
\Bigl(\sum_{j=1}^{m}\sigma(p_j)\Bigr)k=\Bigl(\prod_{s\in F_{0}}\sigma(s)\sigma(s)^{-1}\Bigr)k
\]
then
\begin{align*}
|V| & \geq\left(1-7\delta| F_{\pm}^n|-7\delta| F_{\pm}^n|-\delta| F_{\pm}^n|-\delta| F_{\pm}^n|-157(3^{2n})| F_{\pm}|^{2n}\delta\right)d\\
 & \geq\left(1-175(3^{2n})| F_{\pm}|^{2n}\delta\right)d
\end{align*}
So, for all $\pi_j$ we have that $\sigma(p_j)_{|V}$
is a projection and 
\[
\pi_j(s)=\pi_j(t)\Rightarrow\sigma(s)\sigma(t)^{-1}\sigma(p_j)_{|V}=\sigma(p_j)_{|V}
\text{ and }\tr\left(\sigma(s)\sigma(t)^{-1}\sigma(p_j)_{|V}\right)=0
\]
Hence $\ran(\sigma(p_j) p_V)\subset \ran(p_j^\sigma p_V)$,
but since we also have that 
\[
\Bigl(\sum_{j=1}^{m}\sigma(p_j)\Bigr)p_{V}=\Bigl(\prod_{s\in F_{0}}\sigma(s)\sigma(s)^{-1}\Bigr)p_{V}=\Bigl(\sum_{j=1}^{m}p_j^\sigma\Bigr)p_{V}
\]
we conclude $\sigma(p_j)p_{V}=p_j^\sigma p_{V}$
which implies $\left|\tr\circ\sigma(p_j)-|F^{\sigma}_{0|\pi_j}|/d\right|<175(3^{2n})| F_{\pm}|^{2n}\delta$
therefore
\begin{align*}
\left|\mathfrak{h}(F_0^{\pi_j})-{|F^{\sigma}_{0|\pi_j}|}/{d}\right| & =\left |\tr(p_j)-{|F^{\sigma}_{0|\pi_j}|}/{d}\right|\\
 & \leq|\tr(p_j)-\tr\circ \sigma(p_j)|+\left|\tr\circ \sigma(p_j)-|F^{\sigma}_{0|\pi_j}|/d\right|\\
 & \le\delta+175(3^{2n})| F_{\pm}|^{2n}\delta\leq176(3^{2n})| F_{\pm}|^{2n}\delta.
\end{align*}
\end{proof}

\section{Bernoulli shifts and random partitions} \label{S - Bowen groupoids}

 For $1\leq i\leq q$ let
\[
B_{i}:=\{x\in  \{1,\ldots, q\}^G\mid x_e=i\}
\]
where $X=\{1,\ldots,q\}^G$ is endowed with the Bernoulli action $G\acts X$,
and for $d\in \IZ_{\geq 1}$ let $A_1,\ldots, A_q$ be a partition of $\{1,\ldots, d\}$.
Given a set $F\subset \bb G$ and function $\psi:F\rightarrow\{1,\ldots,q\}$
we define 
\[
B_{\psi}:=\bigcap_{s\in F}sB_{\psi(s)}
\]
and
\[
A_{\psi}:=\bigcap_{s\in F_{\psi}}\sigma(s)A_{\psi(s)}
\]
where $\sigma : F \to \bb d$.

 The next lemma  adapts the
proof of Theorem 8.1 from L. Bowen's paper \cite{Bowen2010}.

\begin{lemma} \label{L -random partitions} If $d$ is large enough (depending on $F$, $n$ and $\delta$) then there is a partition $\{A_{1},A_{2},\ldots,A_{q}\}$ of $\{1\,\ldots d\}$ such that if $\sigma\in \SA(F_{n},4n|F_{\pm}^n|+1,\delta,d)$
then for every subset $F_{\psi}\subset F_\pm^n$ and every function $\psi:F_{\psi}\rightarrow\{1,\ldots, q\}$
\[
\left|\mu(B_{\psi})-\frac{|A_{\psi}|}{d}\right|<c_2(F,n){\delta}
\]
where $c_2(F,n)=2c_1(F,n)\B(|F_{\pm}|^{n})$ and $\B$ is the Bell number.
\end{lemma}
\begin{proof}
Fix $F_{\psi}\subset F_{\pm}^n$ and $\psi:F_{\psi}\rightarrow\{1,\ldots, q\}$. Create a random partition 
\[
\{A_{1},A_{2},\ldots,A_{q}\}
\] 
of
$\{1\,\ldots d\}$  using the following scheme: for each $k\in\{1,\ldots,d\}$ place $k$
in $A_{i}$ with probability $\mu_{0}(i)$. 

We will find the probability that $|\mu(B_{\psi})-|A_{\psi}|/d|<c_2(F,n)\delta$.

Let $\pi_{1},\pi_{2},\ldots, \pi_{m}$ be the family of partitions of $F_{\psi}$. Let $s_j: \ran \pi_j \to F_\psi$ be an arbitrary section of $\pi_j$. If
\[
\chi_j:=\begin{cases}
0 &\text{if }    \exists s, t\in F_\psi\text{ such that }\pi_j(s)=\pi_j(t)\text{ and }\psi(s)\neq \psi(t)\\
1 &\text{otherwise}
\end{cases}
\]
then 
\begin{align*}
\mu(B_{\psi}) & =\mu\left (\bigcap_{s\in F_{\psi}}sB_{\psi(s)}\right)\\
 & =\sum_{j=1}^{m}\chi_j\mathfrak{h}(F_{\psi|\pi_j})\prod_{r\in \ran \pi_j}\mu_{0}\left (\psi(s_j(r))\right)
\end{align*}
Since for $\sigma\in \SA(F_{n},4n|F_{\pm}^n|+1,\delta,d)$, we have 
\[
\frac{|A_{\psi}|}{d}={\displaystyle \sum_{j=1}^{m}}\frac{|A_{\psi}\cap F_{\psi|\pi_j}^{\sigma}|}{d}
\] 
it will suffice to show by the triangle inequality  that for all $j$
\[
\left |\chi_j\mathfrak{h}(F_{\psi|\pi_j})\prod_{k\in \ran \pi_j}\mu_{0}\left (\psi(s_j(k))\right)-\frac{|A_{\psi}\cap F_{\psi|\pi_j}^{\sigma}|}{d}\right|<2c_1(F,n){\delta}.
\]

First suppose $\chi_j=0$, so for some $s, t\in F_\psi$,  $\pi_j(s)=\pi_j(t)$ and $\psi(s)\neq \psi(t)$. 
Then $A_{\psi}\cap F_{\psi|\pi_j}^{\sigma}=\emptyset$, for if $k\in A_{\psi}\cap F_{\psi|\pi_j}^{\sigma}$ then $k\in\sigma(s)A_{\psi(s)}\cap\sigma(t)A_{\psi(t)}\cap F_{\psi|\pi_j}^{\sigma}$
which implies $\sigma(s)^{-1}k=\sigma(t)^{-1}k\in A_{\psi(s)}\cap A_{\psi(t)}$, a
contradiction.

Next suppose $\chi_j=1$. 
For $1\leq k\leq d$ let 
\[
Z_{k}=\begin{cases}
1 & \text{if } k\in F_{\psi|\pi_j}^{\sigma}\cap A_{\psi}\\
0 & \text{otherwise}
\end{cases}
\] 
We wish to compute $\mathbb{E}(Z_{k})$.

For $k\notin F_{\psi|\pi_j}^{\sigma}$ we have $\mathbb{E}(Z_{k})=0$. Otherwise,
$k\in F_{\psi|\pi_j}^{\sigma}\cap A_{\psi}$ if and only if $\sigma_{s_j(r)}^{-1}k\in A_{\psi(s_j(r))}$
for all $r\in \ran \pi_j$ and since 
\[
\sigma_{s_j(r)}^{-1}k\neq\sigma_{s_j(r')}^{-1}k,\ \forall r\neq r'\in \ran \pi_j
\]
we obtain
\[
\mathbb{E}(Z_{k})=\prod_{r\in \ran \pi_j}\mu_{0}(\psi(s_j(r)))
\]
Thus if we let $Z=\sum_{k=1}^{d}Z_{k}=|F_{\psi|\pi_j}^{\sigma}\cap A_{\psi}|$
then 
\[
\mathbb{E}(Z)=|F_{\psi|\pi_j}^{\sigma}|\prod_{r\in \ran \pi_j}\mu_{0}(\psi(s_j(r))).
\]

Now lets bound $\var(Z)$: 
\[
\var(Z)=\mathbb{E}(Z^{2})-\mathbb{E}(Z)^{2}=\sum_{k,l\in\{1\ldots,d\}}\mathbb{E}(Z_{k}Z_{l})-\mathbb{E}(Z)^{2}
\]
For $k,l\notin F_{\psi|\pi_j}^{\sigma}$ we have $\mathbb{E}(Z_{k}Z_{l})=0=\mathbb{E}(Z_{k})\mathbb{E}(Z_{l})$.
On the other hand if $k,l\in F_{\psi|\pi_j}^{\sigma}$ then $Z_k$ and $Z_l$ are not independent
if $\sigma_{s_j(r)}^{-1}k=\sigma_{s_j(r')}^{-1}l$ for some
$r,r'$. Thus are at most $| F_{\pm}^n|^{2}|F_{\psi|\pi_j}^{\sigma}|$
non independent pairs $(k,l)$. Now clearly for these pairs $\mathbb{E}(Z_{k}Z_{l})\leq\mathbb{E}(Z_{k})\mathbb{E}(Z_{l})+1$.
So returning to our equation above
\[
\sum_{1\leq k,l\leq d}\mathbb{E}(Z_{k}Z_{l})-\mathbb{E}(Z)^{2}\leq\sum_{1\leq k,l\leq d}\mathbb{E}(Z_{k})\mathbb{E}(Z_{l})+| F_{\pm}^n|^{2}|F_{\psi|\pi_j}^{\sigma}|-\mathbb{E}(Z^{2})
\]
\[
=\mathbb{E}(Z^{2})+| F_{\pm}^n|^{2}|F_{\psi|\pi_j}^{\sigma}|-\mathbb{E}(Z^{2})=| F_{\pm}^n|^{2}|F_{\psi|\pi_j}^{\sigma}|
\]

Now we can apply Chebyshev's inequality to $\frac{Z}{d}$ for $a>0$
\[
\Pr\left (\left |\frac{Z}{d}-\frac{\mathbb{E}(Z)}{d}\right |\geq a \right)\leq\frac{\var(\frac{Z}{d})}{a^{2}}\leq\frac{| F_{\pm}^n|^{2}|F_{\psi|\pi_j}^{\sigma}|}{a^{2}d^{2}}\leq\frac{| F_{\pm}^n|^{2}}{a^{2}d}
\]
Since $Z=|A_{\psi}\cap F_{\psi|\pi_j}^{\sigma}|$
and
$\mathbb{E}(Z)=|F_{\psi|\pi_j}^{\sigma}|\prod_{r\in \ran \pi_j}\mu_{0}(\psi(s_j(r)))$,
and, by Lemma \ref{L - First Lemma}, 
\[
 \left |\mathfrak{h}(F_{{\psi}|{\pi_j}})-\frac{|F_{\psi|\pi_j}^{\sigma}|}{d}\right|<c_1(F,n)\delta
\]
we have
\[
\Pr\left (\left|\mathfrak{h}(F_{{\psi}|{\pi_j}})\prod_{k\in \ran \pi_j}\mu_{0}\left (\psi(s_j(k))\right)-|F_{{\psi}|{\pi_j}}^\sigma\cap A_{\psi}|/d\right|\geq a+c_1(F,n)\delta\right )\leq\frac{| F_{\pm}^n|^{2}}{a^{2}d}
\]
Let $a=c_1(F,n)\delta$ so then
\[
\Pr\left (\left|\mathfrak{h}(F_{{\psi}|{\pi_j}})\prod_{k\in \ran \pi_j}\mu_{0}\left (\psi(s_j(k))\right)-|F_{{\psi}|{\pi_j}}^\sigma\cap A_{\psi}|/d\right|\geq2c_1(F,n)\delta\right)\leq\frac{| F_{\pm}^n|^{2}}{(c_1(F,n)\delta){}^{2}d}
\]
So for $d$ large enough there is some partition $(A_{1},\ldots, A_{q})$
such that 
\[
\left |\mathfrak{h}(F_{{\psi}|{\pi_j}})\prod_{k\in \ran \pi_j}\mu_{0}\left (\psi(s_j(k))\right)-|F_{{\psi}|{\pi_j}}^\sigma\cap A_{\psi}|/d\right|<2c_1(F,n)\delta
\]
Thus, for $d$ large enough there will be a partition so that this true
for all $j$ so that 
\[
\left|\frac{|A_{\psi}|}{d}-\mu(B_{\psi})\right|<{c_2(F,n)} \delta.
\]
Indeed for large enough $d$ nearly all partitions will satisfy this
and hence we will have non zero probability that for all $F_{\psi}\subset F_{\pm}^n$ and $\psi:F_{\psi}\rightarrow\{1,\ldots, q\}$
the inequality $|\mu(B_{\psi})-|A_{\psi}|/d|<{c_2(F,n)} \delta$
holds. 
\end{proof}

\section{A lemma on approximate equivariance}\label{S - linear}
Given a groupoid $G$ acting on a set $X$ for any finite set of projections
$P\subset L^{\infty}(X,\mu)$  and finite set $F\subset \bb G$ then $P_F$ denotes the set
of all projections of the form 
\[
\prod_{s\in F'}sp_{s}
\]
 where $p_{s}\in P$ and $F'\subset F$.  
 
 Let $P=\{p_{B_i}\}$ where 
 \[
 B_{i}:=\{x\in  \{1,\ldots, q\}^G\mid x_e=i\}
 \] 
and $X=\{1,\ldots,q\}^G$ is endowed with the Bernoulli action $G\acts X$. Fix $F\subset \bb G$
and a basis $\{p_{B_{\psi_{1}}},p_{B_{\psi_{2}}}...p_{B_{\psi_{\ell}}}\}$
for $\vspan( P_{F_{\pm}^n} )$ in $L^\infty(X,\mu)$ associated with $\psi_i : F_{\psi_i}\to \{1,\ldots,q\}$, $1\leq i\leq \ell$. 

We set
\[
\kappa=\max_{\psi,\psi_1,\ldots,\psi_\ell, a_1,\ldots a_\ell\mid p_{B_\psi}=\sum_{i=1}^{\ell}a_{i}p_{B_{\psi_{i}}}}|a_i|\geq 1.
\]
Both $\kappa$ and $\ell$ depend on $F$ and $n$ only.

\begin{lemma}\label{L - approx equiv} Let $\sigma\in \SA(F_{n},4n|F_{\pm}^n|+1,\delta,d)$ and take a partition
$\{A_{1},\ldots,A_{q}\}$ satisfying the conclusion of Lemma \ref{L -random partitions}. 

If for some $\psi:F_{\psi}\rightarrow\{1,\ldots,q\}$ 
\[
p_{B_{\psi}}=\sum_{i=1}^{\ell}a_{i}p_{B_{\psi_{i}}}
\]
then
\[
\left \|p_{A_{\psi}}-\sum_{i=1}^{\ell}a_{i}p_{A_{\psi_{i}}}\right\|_{2}<c_3(F,n)\sqrt{\delta}
\]
where $c_3(F,n)=((1+(3+\kappa\ell )\B(\ell) {N_\ell}))c_2(F,n)$ and $N_\ell=\ell {2^\ell}$.
\end{lemma}
\begin{proof}
For $I\subset\{1,\ldots\ell\}$ define 
\[
D_I:=\{x\in X\mid x\in B_{\psi_{i}}\iff i\in I\}
\]
and similarly
\[
D_I^\sigma:=\{1\leq k\leq d\mid k\in A_{\psi_{i}}\iff i\in I\}.
\]
We want to bound $\left|\mu(D_I)-|D_I^\sigma|/d\right|$.

Note that for any $I\subset\{1,\ldots\ell\}$
\[
\mu\{x\in X\mid \forall i\in I,\ x\in B_{\psi_{i}}\}=\mu\left(\bigcap_{i\in I}B_{\psi_{i}}\right).
\]
Now either for some $s\in F_{\psi_{i_{1}}}\cap F_{\psi_{i_{2}}}$
we have $\psi_{i_{1}}(s)\neq\psi_{i_{2}}(s)$ and so $\bigcap_{i\in I}B_{\psi_{i}}=\emptyset$
or otherwise $\bigcap_{i\in I}B_{\psi_{i}}=B_{\psi'}$ for some $\psi':F_{\psi'}\rightarrow\{1,\ldots,q\}$.
In the first case $\bigcap_{i\in I}A_{\psi_{i}}=\emptyset$ as well
and in the second case by Lemma \ref{L -random partitions} 
\[
|\mu(B_{\psi'})-|A_{\psi'}|/d|<c_2(F,n)\delta
\]
and 
\[
|A_{\psi'}|/d=\left |\bigcap_{i\in I}A_{\psi_{i}}\right|/d=|\{1\leq k\leq d\mid \forall i\in I,\ k\in A_{\psi_{i}}\}|/d
\]
So in general we conclude that $\mu\{x\in X\mid \forall i\in I, \ x\in B_{\psi_{i}}\}$
is within ${c_1(F,n)} \delta$ of $|\{1\leq k\leq d\mid \forall i\in I, \ k\in A_{\psi_{i}}\}|/d$.
By inclusion-exclusion
\[
\mu(D_{I})=\sum_{k=|I|}^{\ell}\sum_{\begin{smallmatrix}
I'\supset I\\
|I'|=k
\end{smallmatrix}}(-1)^{k-|I|}\mu(\{x\in X\mid \forall i\in I',\ x\in B_{\psi_{i}}\})
\]
and similarly
\[
\frac{|D_{I}^\sigma|}{d}=\sum_{k=|I|}^{\ell}\sum_{\begin{smallmatrix}
I'\supset I\\
|I'|=k
\end{smallmatrix}}(-1)^{k-|I|}\frac{1}{d}|\{1\leq k\leq d\mid \forall i\in I',\ k\in A_{\psi_{i}}\}|
\]
so by applying the triangle
inequality at most $ {N_\ell}=\ell{2^\ell}$ times we have 
\[
|\mu(D_I)-|D_I^\sigma|/d|\leq{c_2(F,n)} {N_\ell}\delta.
\]
Now
\[
\mu(B_{\psi})=\sum_{I\subset\{1,\ldots\ell\};\sum_{i\in I}a_{i}=1}\mu(D_I).
\]
So if $\sum_{i\in I}a_{i}=1$ then $\mu\left(\bigcap_{i\in I}B_{\psi_{i}}\right)=\mu\left(\bigcap_{i\in I}B_{\psi_{i}}\cap B_{\psi}\right)$
so by a similar argument $\mu\{x\in B_{\psi}\mid \forall i\in I,\ x\in B_{\psi_{i}}\}$
is within ${c_2(F,n)} \delta$ of $|\{k\in A_{\psi}\mid \forall i\in I, \ k\in A_{\psi_{i}}\}|/d$
So if we let 
\[
\tilde D_{I}^\sigma:=\{k\in A_{\psi}|x\in A_{\psi_{i}}\iff i\in I\}
\]
then
\[
|\mu(D_I)-|\tilde D_{I}^\sigma|/d|\leq{c_2(F,n)} {N_\ell}\delta
\]
so
\[
\left|\frac{|D_I^\sigma|}d-\frac{|\tilde D_{I}^\sigma|}d\right|\leq2{c_2(F,n)} {N_\ell}\delta.
\]
Let $a(I):=\sum_{i\in I}a_{i}$.
\begin{align*}
  &\left\|p_{A_{\psi}}-\sum_{i=1}^{\ell}a_{i}p_{A_{\psi_{i}}}\right\|_{2}\\&\hspace{1cm} =\left\|p_{A_{\psi}}-\sum_{I\subset\{1,\ldots\ell\}}(\sum_{i\in I}a_{i})p_{D_I^\sigma}\right\|_{2}\\&\hspace{1cm} \leq\left\|p_{A_{\psi}}-\sum_{a(I)=1}p_{D_I^\sigma}\right\|_{2}+\left\|\sum_{a(I)\notin\{0,1\}}(\sum_{i\in I}a_{i})p_{D_I^\sigma}\right\|_{2}\\& \hspace{1cm}\leq\left\|p_{A_{\psi}}-\sum_{a(I)=1}p_{\tilde D_{I}^\sigma}\right\|_{2}+\left\|\sum_{a(I)=1}p_{\tilde D_{I}^\sigma}-\sum_{a(I)=1}p_{D_I^\sigma}\right\|_{2}+\kappa \ell\sum_{a(I)\notin\{0,1\}}|D_I^\sigma|/d\\&\hspace{1cm} \leq\sqrt{\left |\frac{|A_{\psi}|}{d}-\sum_{a(I)=1}\frac{|\tilde D_{I}^\sigma|}{d}\right|}+\sum_{a(I)=1}\sqrt{\left |\frac{|\tilde D_{I}^\sigma|}{d}-\frac{|D_I^\sigma|}{d}\right |}+\kappa\ell \B(\ell){c_2(F,n)} \delta{N_\ell}\\&\hspace{1cm} \leq\sqrt{\left|\frac{|A_{\psi}|}{d}-\mu(B_{\psi})\right|+\left|\sum_{a(I)=1}\mu(D_I)-\sum_{a(I)=1}\frac{|\tilde D_{I}^\sigma|}{d}\right|}\\&\hspace{5cm}+\sqrt{\B(\ell)2{c_2(F,n)} \delta{N_\ell}}+\kappa \ell\B(\ell){c_2(F,n)} \delta{N_\ell}\\&\hspace{1cm} \leq\sqrt{{c_2(F,n)\delta}+\sum_{a(I)=1}\left |\mu(D_I)-\frac{|\tilde D_{I}^\sigma|}{d}\right|}+\sqrt{\B(\ell)2{c_2(F,n)} \delta{N_\ell}}+\kappa \ell\B(\ell){c_2(F,n)} \delta{N_\ell}\\&\hspace{1cm} \leq{c_2(F,n)} \sqrt{\delta}+(3+\kappa\ell)Be(\ell) {N_\ell}{c_2(F,n)} \sqrt{\delta}\\&\hspace{1cm}={c_3(F,n)}\sqrt{\delta}.\\
\end{align*}
\end{proof}

\section{Sofic dimension and groupoid actions}\label{S -HA}

We now briefly recall the group action formulation of $s(G\ltimes X)$ given in \cite[Section 5]{DKP2} rephrasing it here in the framework of groupoid actions.

Let $G\acts X$ be a pmp action of a pmp groupoid. Let $1\in F\subset \bb G$ and  
$P$ be a partition of $X$.
We write $\HA(F,P,n,\delta,d)$ for the set of all
pairs $(\sigma,\varphi)$ where $\sigma\in \SA(F,n,\delta,d)$ and
$\varphi$ is a map 
\[
\varphi\colon \bsigma P_{F_{\pm}^n}\to \bb d
\]
defined on $\bsigma P_{F_{\pm}^n}\subset \bb G$ and satisfying

\begin{enumerate}[(i)]
\item $|\tr\circ\varphi(p)-\mu(p)|<\delta$ for all $p\in P_{F_{\pm}^n} $

\item $|\varphi\circ s(p)-\sigma(s)\circ\varphi(p)|<\delta$ for
all $p\in P$ and $s\in F_{\pm}^n$

\item $|\varphi(p_{1}p_{2})-\varphi(p_{1})\varphi(p_{2})|<\delta$
for all $p_{1},p_{2},p_{1}p_{2}\in \bsigma P_{F_{\pm}^n} $.

\item $|\varphi(p_X)-p_{d}|<\delta$, where $p_d:=p_{\{1,\ldots,d\}}$. 
\end{enumerate}
(Note that since $P$
partitions $X$ and $F$ contains the identity we have $p_{X}\in \vspan P_{F_{\pm}^n}$.) 

We  observe that if $\p$ satisfies these conditions then it is automatically approximately linear: 

\begin{lemma} If $(\sigma,\varphi)\in \HA(F,4n,\delta,d)$ then
\[
|\varphi(p_{1} + p_{2}) - (\varphi(p_{1})+\varphi(p_{2}))|<146\delta
\] 
for all $p_{1},p_{2},p_{1}+p_{2} \in \bsigma P_{F_{\pm}^n}$  with $p_1p_2=0$ and $\varphi(p_{1})+\varphi(p_{2})$ is defined as in \cite[Def. 3.3]{DKP1}.
\end{lemma}

\begin{proof}
Using $(p_1+p_2)p_i=p_i$ we have
\[
|\p(p_i)-\p(p_1+p_2)\p(p_i)|<\delta
\] 
for $i=1,2$. Let $\pi_i(\p(p_1),\p(p_2))$ be defined as in \cite[Def. 3.3]{DKP1} so
\[
\p(p_1)+\p(p_2):=\p(p_1)\pi_1(\p(p_1),\p(p_2))+\p(p_2)\pi_2(\p(p_1),\p(p_2)).
\]
By \cite[Lemma 3.4]{DKP1} using the approximate homomorphism property of $\p$ we obtain
\[
|\pi_i(p_1,p_2)-\p(p_i)\p(p_i)^{-1}|<40\delta
\]
for $i=1,2$. Since $|\p(p_i)-\p(p_i)\p(p_i)^{-1}|<8\delta$ we obtain
\[
|\varphi(p_{1} + p_{2})\pi_i(p_1,p_2) - \varphi(p_{i})|<48\delta
\]
and so
\[
|\varphi(p_{1} + p_{2})\pi(p_1,p_2) - (\varphi(p_{1})+\varphi(p_{2}))|<48\delta
\] 
with $\pi(p_1,p_2):=\pi_1(\p(p_1),\p(p_2))+\pi_2(\p(p_1),\p(p_2))$. 

Then 
\begin{align*}
\tr(\varphi(p_{1} + p_{2})\pi(p_1,p_2))&=\tr(\varphi(p_{1} + p_{2})\pi_1(p_1,p_2)))+\tr(\varphi(p_{1} + p_{2})\pi_2(p_1,p_2))\\
&>\tr(\varphi(p_{1}))+\tr(\varphi(p_{2}))-96\delta\\
&>\tau(p_1)+\tau(p_2)-98\delta\\
&=\tau(p_1+p_2)-98\delta.
\end{align*}
Therefore
\[
|\varphi(p_{1} + p_{2}) - (\varphi(p_{1})+\varphi(p_{2}))|<146\delta.
\] 
(Similarly one can also show that $\p(p)$ is approximately a projection for every $p\in \bsigma P_{F_{\pm}^n}$.)
\end{proof}

\begin{definition} Given $E,Q, F$, $P$, $n$ and $\delta$ define successively 
\[
s_{E,Q}(F, {P},n,\delta):=\limsup_{d\rightarrow\infty}\frac{1}{d\log(d)}\log(|\HA(F, {P},n,\delta,d)|_{E,Q})
\]
\[
s_{E,Q}(F, {P},n):=\inf_{\delta>0}s_{E,Q}(F, {P},n,\delta)
\]
\[
s_{E,Q}(F, {P}):=\inf_{n\in\mathbb{N}}s_{E,Q}(F, {P},n)
\]
If $K\subset \bb G$ is a transversally generating set and $R$ is a  dynamically generating family of projections of $L^{\infty}(X,\mu)$ define
\[
s(K, R):=\sup_{E}\sup_{Q}\inf_{F}\inf_{ {P}}s_{E,Q}(F, {P})
\]
where $E$ and $F$ range over finite subsets of $K$ and ${P}$, ${Q}$
range over finite subpartitions  of $R$. Set $s(G,X)=s(\bb G, L^\infty(X,\{0,1\}))$. 
\end{definition}

Since $\bb G\cup L^\infty(X,\{0,1\})$ is a transversally generating set of the crossed product groupoid we obtain by Theorem \ref{T - OE invariance} (if the action $G\acts X$ is essentially free) and \cite[Theorem 2.11]{DKP2} in general (compare \cite[Proposition 5.2]{DKP2}):

\begin{proposition}\label{P-s via HA} 
$s(G\ltimes X)=s(G,X).$
\end{proposition}

Moreover if $F$ is a finite transversally generating subset of $\bb G$ and $ {P}$
is a finite and dynamically generating partition of unity, 
then $s(G\ltimes X)=s_{F, {P}}(F, {P})$.

\begin{proposition}\label{P-action general}
$s(G\ltimes X)\leq s(G)$.
\end{proposition}
\begin{proof}
For every $(\sigma,\varphi)\in \HA(F, {P},n,\delta,d)$
there are at most $|{Q}|^{d}$ restrictions $\varphi|_{{Q}}$
and hence 
\[
|\HA(F, {P},n,\delta,d)|_{E,Q}\leq|{Q}|^{d}|\SA(F,n,\delta,d)|_{E}
\]
so we have $s_{E,Q}(F, {P})\leq s_{E}(F)$ and the result follows
directly by the proposition above. 
\end{proof}

\begin{remark}\label{Rem - HA def}
The set $\HA$  differs from the set $\HA$ introduced in  \cite{DKP2} in that we do not assume $\p$ to be a strict homomorphism. Furthermore the maps in \cite{DKP2} are defined on $L^\infty(X)$ with values in $\M_d$ using the 2-norm. It is often convenient to adopt the latter point of view for computational purposes and we will do so below. The definition of $\HA$ given above has the advantage of being  purely finitary, in the spirit of the sofic property.  We have maintained the notation $\HA$ in view of (iii) and (iv). 
\end{remark}

\section{The computation of $s(G\ltimes (X_0,\mu_0)^G)$}\label{S - Bernoulli corresp}

\begin{lemma}\label{P-action Bernoulli} If $G\acts  \{1,\ldots, q\}^G$, then $s(G\ltimes (\{1,\ldots, q\},\mu_0)^G)=s(G)$ for any probability measure $\mu_0$ on $\{1,\ldots, q\}$ and pmp groupoid $G$.  The same holds true for $\underline s$ and $s^\omega$.
\end{lemma}

\begin{proof}
Let $X=\{1,\ldots, q\}^G$. By Proposition \ref{P-action general} we  have to prove $s(G\ltimes X)\geq s(G)$.
Let $P=\{p_{B_{i}}\}$ where $B_{i}=\{x\in  \{1,\ldots, q\}^G\mid x_e=i\}$. Let $\{p_{B_{\psi_{1}}},p_{B_{\psi_{2}}}...p_{B_{\psi_{\ell}}}\}$ be a basis for $\vspan( P_{F_{\pm}^n} )\subset L^\infty(X,\mu)$, where $\psi_i : F_{\psi_j}\to \{1,\ldots,q\}$.  
Let $\sigma\in \SA(F_{n},4n|F_{\pm}^n|+1,\delta,d)$.  Let $\kappa$ be as  defined before Lemma \ref{L - approx equiv}. 
Let 
\begin{align*}
\gamma_{1} & =\min\left \{\left|\sum_{i\in T}a_{i}\right|:\psi,\psi_{1},....\psi_{\ell},a_{1}...,a_{\ell},T\subset\{1,\ldots,\ell\}\ \text{and}\ p_{B_{\psi}}=\sum_{i=1}^{\ell}a_{i}p_{B_{\psi_{i}}}\right\}/\{0\}\\
\gamma_{2} & =\min\left \{\left|\sum_{i\in T}a_{i}-1\right|:\psi,\psi_{1},....\psi_{\ell},a_{1}...,a_{\ell},T\subset\{1,\ldots,\ell\}\ \text{and}\ p_{B_{\psi}}=\sum_{i=1}^{\ell}a_{i}p_{B_{\psi_{i}}}\right\}/\{0\}\\
\gamma_{2} & =\min\left\{ \left|\sum_{i\in T,j\in T'}a_{i}b_{j}\right|:\psi,\psi'\psi_{1},....\psi_{\ell},a_{1}...,a_{\ell},T,T'\subset\{1,\ldots,\ell\}\text{ and}\right.\\
 & \hspace{1em}\hspace{1em}\hspace{1em}\hspace{1em}\left.p_{B_{\psi}}=\sum_{i=1}^{\ell}a_{i}p_{B_{\psi_{i}}},p_{B_{\psi'}}=\sum_{i=1}^{\ell}b_{i}p_{B_{\psi_{i}}}\right\} /\{0\}
\end{align*}

\[
\gamma=\min\{\gamma_{1},\gamma_{2},\gamma_{3}\}
\]
so that $\gamma$ depends only on $F$ and $n$ and $\gamma\leq1$.
 We want to find $\varphi$ such that 
\[
(\sigma,\varphi)\in \HA(F,P,n,9|P_{F_{\pm}^n}|^2\frac{1}{\gamma^2}\kappa^{5}\ell^{2}qc_3(F,n)^2\sqrt{\delta},d)
\]
(for sufficiently large $d$). Using Proposition \ref{P-s via HA}
this will complete the proof.

Take a partition
$\{A_{1},\ldots,A_{q}\}$ such that the conclusion of Lemma \ref{L -random partitions}
holds (namely a random partition for $d$ large). 
For each
$p_{B_{\psi_{i}}}$, $i=1\ldots \ell$, let
\[
\varphi_0(p_{B_{\psi_{i}}}):=p_{A_{\psi_{i}}}
\]
and extend $\varphi_0$ linearly to $\vspan(P_{F_{\pm}^n})\subset L^\infty(X,\mu)$ with values in $\M_d$.  We will check that $\varphi_0$ satisfies the following properties, where $\delta_0=3\frac{1}{
\gamma}\kappa^{2}\ell^{2}qc_3(F,n)^2\sqrt{\delta}$ (compare Remark \ref{Rem - HA def}): 
\begin{enumerate}
\item $|\tr\circ\varphi_0(p)-\mu(p)|<\delta_0$ for all $p\in P_{F_{\pm}^n} $

\item $\left\|\varphi_0\circ s(p)-\sigma(s)\circ\varphi_0(p)\right\|_{2}<\delta_0$ for
all $p\in P$ and $s\in F_{\pm}^n$

\item $\left\|\varphi_0(p_{1}p_{2})-\varphi_0(p_{1})\varphi_0(p_{2})\right\|_{2}<\delta_0$
for all $p_{1},p_{2}\in \vspan P_{F_{\pm}^n} $

\item $\left\|\varphi_0(p_X)-p_{d}\right\|_{2}<\delta_0$, where $p_d:=p_{\{1,\ldots,d\}}$. 

\end{enumerate}
 One can see (as shown below) that properties (1)-(4) are closely related to the properties (i)-(iv) defined above.
\begin{enumerate}
\item 
Note for any $x\in\{1\ldots,d\}$, $\psi$ and $a_{1},...a_{n}$ if
\[
\left|\left(\sum_{i=1}^{\ell}a_{i}p_{A_{\psi_{i}}}-p_{A_{\psi}}\right)x\right|\neq0
\]
then
\[
\left|\left(\sum_{i=1}^{\ell}a_{i}p_{A_{\psi_{i}}}-p_{A_{\psi}}\right)x\right|\geq\min\{\gamma_1,\gamma_2\}\geq\gamma.
\]
Let $p_{B_{\psi}}\in P_{F_{\pm}^{n}}$ say $p_{B_{\psi}}=\sum_{i=1}^{\ell}a_{i}p_{B_{\psi_{i}}}$ then
\[
\varphi_0(p_{B_{\psi}})=\sum_{i=1}^{\ell}a_{i}p_{A_{\psi_{i}}}
\]
so
\begin{align*}
\left|\tr\circ\varphi_0(p_{B_{\psi}})-\mu(p_{B_{\psi}})\right| & \leq\left|\tr(\sum_{i=1}^{\ell}a_{i}p_{A_{\psi_{i}}})-\tr\circ p_{A_{\psi}}\right|+\left|\frac{|A_{\psi}|}{d}-\mu(B_{\psi})\right|\\
 & \leq\frac{1}{\gamma}\left\Vert \sum_{i=1}^{\ell}a_{i}p_{A_{\psi_{i}}}-p_{A_{\psi}}\right\Vert _{2}^{2}+c_{2}(F,n)\delta \\
 & \leq\frac{1}{\gamma}c_{3}(F,n)^{2}\delta^{2}+c_{2}(F,n)\delta \leq\frac{1}{\gamma}c_{3}(F,n)^{2}\sqrt{\delta}\text{, for \ensuremath{\delta<1}}
\end{align*}
where we use both Lemma \ref{L -random partitions} and \ref{L - approx equiv}
\item Here we use Lemma \ref{L - approx equiv}. Let $s\in F_{\pm}^n$
and suppose 
\[
sp_{B_{i}}=\sum_{i=1}^{\ell}a_{i}p_{B_{\psi_{i}}}.
\]
Then
\[
\varphi_0(sp_{B_{i}})=\sum_{i=1}^{\ell}a_{i}p_{A_{\psi_{i}}}
\]
and 
\[
\left\|\sum_{i=1}^{\ell}a_{i}p_{A_{\psi_{i}}}-\sigma(s)p_{A_{i}}\right\|_{2}<c_3(F,n)\sqrt{\delta}.
\]
On the other hand 
\[
p_{B_{i}}=\sum_{i=1}^{\ell}b_{i}p_{B_{\psi_{i}}}\Rightarrow\varphi_0(p_{B_{i}})=\sum_{i=1}^{\ell}b_{i}p_{A_{\psi_{i}}}
\]
and
\[
\left \|\sum_{i=1}^{\ell}b_{i}p_{A_{\psi_{i}}}-\sigma(1)p_{A_{i}}\right\|_{2}<c_3(F,n)\sqrt{\delta}.
\]
Thus
\begin{align*}
&\|\varphi_0(sp_{B_{i}})-\sigma(s)\varphi_0(p_{B_{i}})\|_{2}\\
&\hspace{1cm}\leq\|\varphi_0(sp_{B_{i}})-\sigma(s)p_{A_{i}}\|_{2}+\|\sigma(s)p_{A_{i}}-\sigma(s)\sigma(1)p_{A_{i}}\|_{2}\\
&\hspace{6cm}+\|\sigma(s)\sigma(1)p_{A_{i}}-\sigma(s)\varphi_0(p_{B_{i}})\|_{2}\\
&\hspace{3cm}<2c_3(F,n)\sqrt{\delta}+\sqrt{\delta}=3c_3(F,n)\sqrt{\delta}
\end{align*}

\item Let $p_{1}=\sum_{i=1}^{\ell}a_{i}p_{B_{\psi_{i}}}$
and $p_{2}=\sum_{i=1}^{\ell}b_{i}p_{B_{\psi_{i}}}$ then
\begin{align*}
\|\varphi_0(p_{1}p_{2})-\varphi_0(p_{1})\varphi_0(p_{2})\|_{2} & =\|\sum_{i,j=1}^{\ell}a_{i}b_{i}(\varphi_0(p_{B_{\psi_{i}}}p_{B_{\psi_{j}}})-p_{A_{\psi_{i}}}p_{A_{\psi_{j}}})\|_{2}\\
 & \le\kappa^{2}\sum_{i,j=1}^{\ell}\|\varphi_0(p_{B_{\psi_{i}}}p_{B_{\psi_{j}}})-p_{A_{\psi_{i}}}p_{A_{\psi_{j}}}\|_{2}
\end{align*}

So it is enough to show that for all $i,j$ 
\[
\|\varphi_0(p_{B_{\psi_{i}}}p_{B_{\psi_{j}}})-p_{A_{\psi_{i}}}p_{A_{\psi_{j}}}\|_{2}<c_3(F,n){\sqrt{\delta}}.
\]
But as we have seen in the proof of Lemma \ref{L - approx equiv}, either $p_{B_{\psi_{i}}}p_{B_{\psi_{j}}}$
and $p_{A_{\psi_{i}}}p_{A_{\psi_{j}}}$ are both $0$ in which case
we are done or $p_{B_{\psi_{i}}}p_{B_{\psi_{j}}}=p_{B\psi'}$ and
$p_{A_{\psi_{i}}}p_{A_{\psi_{j}}}=p_{A_{\psi'}}$ for some $\psi'$
in which case if $p_{B_{\psi'}}=\sum_{i=1}^{\ell}c_{i}p_{B_{\psi_{i}}}$then
by Lemma \ref{L - approx equiv}
\[
\|\varphi_0(p_{B_{\psi_{i}}}p_{B_{\psi_{j}}})-p_{A_{\psi_{i}}}p_{A_{\psi_{j}}}\|_{2}=\|\sum_{i=1}^{\ell}c_{i}p_{A_{\psi_{i}}}-p_{A_{\psi'}}\|_{2}<c_3(F,n){\sqrt{\delta}}
\]

\item $p_{X}=\sum_{i=1}^{q}p_{B_{i}}$ and $I_{d}=\sum_{i=1}^{q}p_{A_{i}}$
so by the triangle inequality it suffices to show 
\[
\|\varphi_0(p_{B_{i}})-p_{A_{i}}\|_{2}<2c_3(F,n)\sqrt{\delta}
\]
By  Lemma \ref{L - approx equiv}
\[
\|\varphi_0(p_{B_{i}})-\sigma(1)p_{A_{i}}\|_{2}<c_3(F,n)\sqrt{\delta}
\]
so 
\begin{align*}
\|\varphi_0(p_{B_{i}})-p_{A_{i}}\|_{2}\leq &\|\varphi_0(p_{B_{i}})-\sigma(1)p_{A_{i}}\|_{2}+\|\sigma(1)p_{A_{i}}-p_{A_{i}}\|_{2}\\
&<2c_3(F,n)\sqrt{\delta}+\sqrt{\delta}<2c_3(F,n)\sqrt{\delta}.
\end{align*}
\end{enumerate}

We now show how to define $\p$ using $\p_0$. For each $p_{B_{\psi}}$

\[
\left\Vert \varphi_{0}(p_{B_{\psi}})-p_{A_{\psi}}\right\Vert _{2}\leq c_{3}(F,n)\sqrt{\delta}
\]
by Lemma \ref{L - approx equiv}. Let $M$ be the number of $x\in\{1\ldots,d\}$ such that
$\varphi_{0}(p_{B_{\psi}})x=p_{A_{\psi}}x$ then
\[
\left\Vert \varphi_{0}(p_{B_{\psi}})-p_{A_{\psi}}\right\Vert _{2}\geq\sqrt{\frac{d-M}{d}\gamma^{2}}
\]
hence
\[
M\geq d(1-c_{3}(F,n)^{2}\frac{1}{\gamma^{2}}\delta)
\]
Similarly for $p_{B_{\psi}},p_{B_{\psi'}}$ we have
\[
\left\Vert \varphi_{0}(p_{B_{\psi}})\varphi_{0}(p_{B_{\psi'}})-\varphi(p_{B_{\psi}})\varphi(p_{B_{\psi'}})\right\Vert _{2}\leq2\kappa^{2}c_{3}(F,n)\sqrt{\delta}
\]
by (3) so 
\[
\left|\{x\in[d]:\varphi_{0}(p_{B_{\psi}})\varphi_{0}(p_{B_{\psi'}})x=0\}\right|\geq d(1-c_{3}(F,n)^{2}\kappa^{4}\delta\frac{1}{\gamma^{2}})
\]
Let $V$ be the set of all $x\in\{1,...d\}$ such that for all $p_{B_{\psi}}$
(arbitrarily represented) we have
\[
\varphi_{0}(p_{B_{\psi}})x=p_{A_{\psi}}x
\]
and for all $p_{B_{\psi}},p_{B_{\psi'}}$ with $p_{B_{\psi}}p_{B_{\psi'}}=0$
we have
\[
\varphi_{0}(p_{B_{\psi}})\varphi_{0}(p_{B_{\psi'}})x=0
\]
then
\[
|V|\geq d(1-|P_{F_{\pm}^{n}}|c_{3}(F,n)^{2}\delta\frac{1}{\gamma^{2}}-|P_{F_{\pm}^{n}}|^{2}c_{3}(F,n)^{2}\kappa^{4}\delta\frac{1}{\gamma^{2}})\geq d(1-2|P_{F_{\pm}^{n}}|^{2}c_{3}(F,n)^{2}\kappa^{4}\delta\frac{1}{\gamma^{2}})
\]
and for any $p\in\bsigma P_{F_{\pm}^{n}}$ $\varphi_{0}(p)\big|_{V}\subset \bb d$.
We finally define
\begin{align*}
\varphi:\bsigma P_{F_{\pm}^{n}}&\rightarrow \bb d\\
p&\mapsto\varphi_{0}(p)\big|_{V}
\end{align*}
Let us check that $\varphi$ satisfies (i)-(iv):
\begin{enumerate}[(i)]
\item 
\begin{align*}
|\tr\circ\varphi(p_{B_{\psi}})-\mu(B_{\psi})|&\leq|\tr\circ\varphi(p_{B_{\psi}})-\tr\circ\varphi_{0}(p_{B_{\psi}})|+\delta_0\\
&\leq\kappa\ell 2|P_{F_{\pm}^{n}}|^{2}c_{3}(F,n)^{2}\kappa^{4}\delta\frac{1}{\gamma^{2}}+\delta_0
\end{align*}
\item 
\begin{align*}
|\varphi(sp_{B_{i}})-\sigma(1)\varphi(p_{B_{i}})| & \le\|\varphi(sp_{B_{i}})-\sigma(1)\varphi(p_{B_{i}})\|_{2}\\
 & \leq\|\varphi(sp_{B_{i}})-\varphi_{0}(sp_{B_{i}})\|_{2}\\
 &\hskip2cm+\|\sigma(1)\varphi(p_{B_{i}})-\sigma(1)\varphi_{0}(p_{B_{i}})\|_{2}+\delta_0\\
 & \leq\kappa\ell4|P_{F_{\pm}^{n}}|c_{3}(F,n)\kappa^{2}\sqrt{\delta}\frac{1}{\gamma}+\delta_0
\end{align*}

\item 
\begin{align*}
|\varphi(p_{1})\varphi(p_{2})-\varphi(p_{1}p_{2})| & \leq\|\varphi(p_{1})\varphi(p_{2})-\varphi(p_{1}p_{2})\|_{2}\\
 & \leq\|\varphi(p_{1})\varphi(p_{2})-\varphi_{0}(p_{1})\varphi_{0}(p_{2})\|_{2}\\
 &\hskip2cm +\|\varphi(p_{1}p_{2})-\varphi_{0}(p_{1}p_{2})\|_{2}+\delta_0\\
 & \leq\kappa^{2}\ell^{2}q4|P_{F_{\pm}^{n}}|c_{3}(F,n)\kappa^{2}\sqrt{\delta}\frac{1}{\gamma}+\delta_0
\end{align*}
%\item Similarly $|\varphi(p_{1}+p_{2})-\varphi(p_{1})-\varphi(p_{2})|\leq\kappa\ell6|P_{F_{\pm}^{n}}|c_{3}(F,n)\kappa^{2}\sqrt{\delta}\frac{1}{\gamma}+\delta_0$
\item  
\[
|\varphi(p_{X})-p_{d}|\leq\kappa\ell2|P_{F_{\pm}^{n}}|c_{3}(F,n)\kappa^{2}\sqrt{\delta}\frac{1}{\gamma}+\delta_0
\]
\end{enumerate}

Thus for every $\sigma\in \SA(F_{n},4n|F_{\pm}^n|+1,\delta,d)$
with $d$ sufficiently large we found $\varphi$ so that  $(\sigma,\varphi)\in \HA(F,P,n,9|P_{F_{\pm}^n}|^2\frac{1}{\gamma^2}\kappa^{5}\ell^{2}qc_3(F,n)^2\sqrt{\delta},d)$. 

We now check that
\[
s(G)=\sup_{E}\inf_{F}\inf_{n\in\mathbb{N}}\inf_{\delta>0}\limsup_{d\rightarrow\infty}\frac{1}{d\log d}\left|\SA(F_n,4n|F_{\pm}^n|+1,\delta,d)\right|_{E}
\]
Recall that $F_n:=F\cup \{p_{{F_0}_{|\pi}}\mid F_0\subset  F_{\pm}^n, \, \pi\text{ partition of }F_0\}$.

The right hand side equals
\[
 \sup_{E}\inf_{F}\inf_{n\in\mathbb{N}}s_{E}(F_{n},4n|F_{\pm}^n|+1).
\]
Let $\epsilon>0$ and choose $n_{0}$ so that 
\[
|s_{E}(F_{n_{0}},4n_0|F_{\pm}^{n_0}|+1)-\inf_{n}s_{E}(F_n,4n|F_{\pm}^n|+1)|<\varepsilon
\]
 and 
\[
|S_{E}(F,4n_0|F_{\pm}^{n_0}|+1)-\inf_{n}s_{E}(F,4n|F_{\pm}^n|+1)|<\varepsilon.
\]
Clearly $s_{E}(F_{n_{0}},4n_0|F_{\pm}^{n_0}|+1)\leq s_{E}(F,4n_0|F_{\pm}^{n_0}|+1)$
but on the other hand if $\sigma\in \SA(F_{n_{0}},4n_0|(F_{n_0})_{\pm}^{n_0}|+1,\delta,d)$
then $\sigma\in \SA(F_{n_{0}},4n_0|F_{\pm}^{n_0}|+1,\delta,d)$
since $4n_0|(F_{n_0})_{\pm}^{n_0}|+1\geq 4n_0|F_{\pm}^{n_0}|+1$. So 
\[
\inf_{F}s_{E}(F,4{n_0}|F_{\pm}^{n_0}|+1)=\inf_{F}s_{E}(F_{n_0},4{n_0}|F_{\pm}^{n_0}|+1)
\]
Since $\varepsilon$ was arbitrary 
\[
\inf_{F}\inf_{n}s_{E}(F_{n},4n|F_{\pm}^n|+1)=\inf_{F}\inf_{n}s_{E}(F,4n|F_{\pm}^n|+1)=s_{E}(G)
\]
since
\[
\inf_{n}s_{E}(F,n)=\inf_{n}s_{E}(F,4n|F_{\pm}^n|+1).
\] 

However 
\[
\left|\SA(F_n,4n|F_{\pm}^n|+1,\delta,d)\right|_{E}\leq|\HA(F,P,n,9|P_{F_{\pm}^n}|^2\frac{1}{\gamma^2}\kappa^{5}\ell^{2}qc_3(F,n)^2\sqrt{\delta},d)|_E
\] 
so $s(G)\leq s(G\ltimes X)$ by Proposition \ref{P-s via HA}. Replacing $\limsup$ by $\liminf$ or $\lim_{d\to\omega}$ above, we get a similar inequality for $\underline s$ and $s^\omega$.
\end{proof}

\begin{theorem}
\label{T-action Bernoulli infinite} Let $G$ be a pmp groupoid, $(X_0,\mu_0)$ be a standard probability space, and $G\acts  X_0^G$ be the corresponding a Bernoulli action. Then $s(G\ltimes X_0^G)=s(G)$. The same holds true for $\underline s$ and $s^\omega$.
\end{theorem}
\begin{proof}
Again we only need to prove $s(G)\leq s(G\ltimes X_0^G)$.
Let $U=\{U_{1},U_{2}, \ldots, U_{q}\}$ be any finite partition of $X_{0}$
into measurable sets. Let
\[
B_{U}=\{B_{U_{1}}, \ldots,B_{U_{q}}\}
\]
where
\[
B_{U_{i}}=\{x\in X_{0}^{G}\mid \forall e\in G_{0},\ x_e\in U_{i}\}
\]
Let $M:=\{p_{U_{i}}\}$ for all possible choices of $U$ and all $i$, so that $M$ is dynamically generating {*}-subalgebra of $L^{\infty}(X,\mu)$.
Now if $P_{U}:=\{p_{U_{1}},\ldots, p_{U_{q}}\}$ then
\[
s(G\ltimes X)=\sup_{E}\sup_{Q}\inf_{F}\inf_{P_{U}}s_{E,Q}(F,P_{U})=\sup_{E}\sup_{Q}\inf_{P_{U}}\inf_{F}s_{E,Q}(F,P_{U})
\]
Since the map $X_0\to\{1\ldots,q\}$ defined by $U_i\ni x\mapsto i$ extends $G$-equivariantly to $X_0^G\to \{1,\ldots,q\}^G$ we have  
\[
\inf_{F}s(F)\leq\inf_{F}s_{E,Q}(F,P_{U}).
\]
Hence $s(G)\leq s(G\ltimes X).$ The same holds for $\underline s$ and $s^\omega$.
\end{proof}

\begin{corollary}\label{C - Shifts regularity}
 A pmp groupoid G is $s$-regular if and only if the crossed product groupoid $G\ltimes X_0^G$ associated with the Bernoulli action $G\acts X_0^G$ is $s$-regular for any base space $(X_0,\mu_0)$. 
\end{corollary}

\begin{proof} Since
\[
\underline s(G)=\underline s(G\times X_0^G)\leq s(G\times X_0^G)=s(G)
\]
$\underline s(G)=s(G)$ if and only if $\underline s(G\times X_0^G)= s(G\times X_0^G)$.
\end{proof}

\section{The scaling formula}\label{S -Scaling}

\begin{proposition}[Scaling formula]\label{P-scaling}
Let $G$ be an ergodic pmp groupoid, $0\neq p\in L^\infty(G^0)$. Then 
\[
s(G)-1=\h(p)(s(pGp)-1)
\]
where $pGp$ is endowed with the normalized Haar measure $\frac {\h_{|pGp}} {\h(p)}$. Furthermore the same equality holds for $\underline s$ and $s_\omega$ therefore: 

\centerline{$G$ is $s$-regular $\ssi$ $pGp$ is $s$-regular.}
\end{proposition}

We will start by showing $\geq$ in the rational case, which is easier to handle:

\begin{lemma}\label{L-scaling - rational} If $\h(p)\in \IQ$ then 
\[
s(pGp)\le\frac{1}{\mathfrak{h}(p)}s(G)+1-\frac{1}{\mathfrak{h}(p)}.
\]
Furthermore the same inequality holds for $\underline s$ and $s_\omega$.
\end{lemma}

\begin{proof} Write $\h(p)=\frac {N-k} N$ and choose $s_{1},s_{2}...s_{k}\in\bb G$ with 
\[
s_{i}^{-1}s_{i}\leq p, \ \h(s_{i}^{-1}s_{i})=\frac 1 N
\]
and 
\[
\sum_{i=1}^{k}s_{i}s_{i}^{-1}=1-p.
\] 
Let
$S=\{s_{1},...s_{k}\}$ so $pGp\cup S$ is generating $G$ and let $E,F\subset\bb{pGp}$ be finite subsets. We may assume that $p\in E\cap F$ and $s_{i}^{-1}s_{i}\in E\cap F$
for all $i$. Let $\sigma\in \SA_{pGp}(F,4n+5,\delta,d)$. In light of \cite[Lemma 2.13]{DKP2} we may also assume that $N-k|d$. Let $d'=\frac{N}{N-k}d$ and partition $\{1,\ldots, d'\}$ into sets $A_{0}...A_{k}$
with $A_{0}=\{1,...d\}$ and $|A_{i}|=\frac {d'} N$. 
Since  
\[
|\Fix\big(\sigma(s_{i}^{-1}s_{i})\big)|\geq\frac{\mathfrak{h}(s_{i}^{-1}s_{i})}{\mathfrak{h}(p)}d-\delta d
\]
 we may choose a subset $B_i\subset A_0$ with exactly $\frac{d'}{N}$ elements such that
\[
\left|\Fix\big(\sigma(s_{i}^{-1}s_{i})\big)\bigtriangleup B_i\right|<\delta d'.
\] 
For each $i=1\ldots k$ let $\gamma(s_i)$ be a bijection $B_i\to A_i$ (we have $\left(\frac {d'} N!\right)^k$ choices).

As
\[
\begin{array}{c}
s_{j}^{-1}s_{i}^{-1}=s_{i}s_{j}=0\ i,j\geq 1\\
s_{i}^{-1}s_{j}=0\ i\neq j\geq 1\\
ss_{i}=s_{i}^{-1}s=0\ i\geq 1,s\in pGp\\
s_{i}^{-1}s_{i}\in F
\end{array}
\]
each element of $\bigcup_{n}(\bb{pGp}\cup S)^{n}$
can be written (not necessarily uniquely) as $s_{i}fs_{j}^{-1}$ for $i,j\geq 0$ where $s_0=p$ and $f\in \bb{pGp}$. Let $\gamma(s_0)=p_{A_0}$
We define a map 
\[
\sigma_\gamma\colon (F\cup S)_\pm^{2n+5}\to \bb {d'}
\] 
by
\[
\sigma_\gamma(s_ifs_j^{-1}):= \gamma(s_i)\sigma'(s_i^{-1}s_ifs_j^{-1}s_j) {\gamma(s_j)}^{-1}.
\]
where we denote  $\sigma'(f)$ the permutation $\left(\begin{smallmatrix}\sigma(f) & 0\\
0 & 0
\end{smallmatrix}\right)
$ acting in $A_0\subset \{1,\ldots,d'\}$. Note that $\sigma_\gamma$ is well defined: if $s_{i_{1}}f_{1}s_{j_{1}}^{-1}=s_{i_{2}}f_{2}s_{j_{2}}^{-1}$
then since the $s_{i}$ have disjoint ranges we have $i_{1}=i_{2}$
$j_{1}=j_{2}$ so $s_{i_{1}}f_{1}s_{j_{1}}^{-1}=s_{i_{1}}f_{2}s_{j_{1}}^{-1}$
hence $s_{i_{1}}^{-1}s_{i_{1}}f_{1}s_{j_{1}}^{-1}s_{j_{1}}=s_{i_{1}}^{-1}s_{i_{1}}f_{2}s_{j_{1}}^{-1}s_{j_{1}}$
so that 
\[
\gamma(s_{i_{1}})\sigma'(s_{i_{1}}^{-1}s_{i_{1}}f_{1}s_{j_{1}}^{-1}s_{j_{1}}){\gamma(s_{j_{1}})}^{-1}=\gamma(s_{i_{1}})\sigma'(s_{i_{1}}^{-1}s_{i_{1}}f_{2}s_{j_{1}}^{-1}s_{j_{1}})\gamma(s_{j_{1}})^{-1}.
\]
then extend $\sigma_\gamma$ linearly to $\bsigma (F\cup S)_\pm^{2n+5}$.

Let $\delta'=5N^2\delta+150N^2(2|F\cup S|+1)^{2(2n+5)}\delta$. We claim that: 
\[
{\sigma_\gamma}\in \SA_{G}(F\cup S,n,\delta',d).
\] 

Let us first see what this implies. Note if $\sigma^1(f)\neq\sigma^2(f)$ then $\sigma^1_{\gamma}(f)=p_{A_0}\sigma^1_{\gamma}(f)p_{A_0}=\sigma^1_{\gamma}(f)\neq\sigma^2_{\gamma}(f)$. Next note if $\gamma_{1}(s_{i})\neq\gamma_{2}(s_{i}):B_{i}\rightarrow A_{i}$
are two bijections then if we have 
\[
\sigma_{\gamma_{1}}(s_{i})=\gamma_{1}(s_{i})\sigma'(p)=\gamma_{2}(s_{i})\sigma'(p)=\sigma_{\gamma_{2}}(s_{i})
\]
we must have $\gamma_{1}(s_{i})\big|_{fix(\sigma(p))}=\gamma_{2}(s_{i})\big|_{fix(\sigma(p))}$.
But since $|\Fix(\sigma(p))\triangle\{1,..,.d\}|\leq\delta d$ we have
at most $\left|\bb{\lceil\delta d\rceil}\right|\leq(\lceil\delta d\rceil+1)^{\lceil\delta d\rceil}$
choices for $\gamma_{2}$ such that $\gamma_{2}(s_{i})\neq\gamma_{1}(s_{i})$
and $\sigma_{\gamma_{2}}(s_{i})=\sigma_{\gamma_{1}}(s_{i})$. Therefore
there are at most $(\lceil\delta d\rceil+1)^{\lceil\delta d\rceil k}$
choices for $\gamma_{2}$ such that $\sigma_{\gamma_{1}}\big|_{E\cup S}\neq\sigma_{\gamma_{2}}\big|_{E\cup S}$
 Thus each $\sigma_{|E}\in \SA_{pGp}(F,n,\delta,d)$
gives at least $\frac{\left(\frac {d'} N!\right)^k}{(\lceil\delta d\rceil+1)^{\lceil\delta d\rceil k}}$ distinct elements 
${\sigma'}_{|E\cup S}\in \SA_{G}(F\cup S,n,\delta',d')$ so we obtain 
\[
|\SA_{pGp}(F,n,\delta,d)|_{E}\leq |\SA_{G}(F\cup S,4n+5,\delta',d')|_{E\cup S}/\frac{\left(\frac {d'} N!\right)^k}{(\lceil\delta d\rceil+1)^{\lceil\delta d\rceil k}}
\]

So let us show that ${\sigma_\gamma}\in \SA_{G}(F\cup S,n,\delta',d')$. We first show that $\sigma_\gamma$ is approximately linear. Suppose in the span we have 
\[
s_{i}fs_{j}=\sum_{\ell=1}^{M}s_{i_{\ell}}f_{\ell}s_{j_{\ell}},
\]
then if we let $L(u,v)$ be the set of all $\ell$ with
$i_{\ell}=u$ $j_{\ell}=v$ then $s_{i}fs_{j}=\sum_{\ell\in L(i,j)}s_{i}f_{\ell}s_{j}^{-1}$
and $s_{u}0s_{v}=0=\sum_{\ell\in L(u,v)}s_{p}f_{\ell}s_{q}^{-1}$
for $(u,v)\neq(i,j)$. So we may assume $i_\ell=i,j_\ell=j$ for all $\ell$ then
\begin{align*}
\sigma_\gamma(s_{i}fs_{j})&=\gamma(s_{i})\sigma'(s_{i}^{-1}s_{i}fs_{j}^{-1}s_{j})\gamma(s_{j})^{-1}\\
&=\gamma(s_{i})\sigma'\left (s_{i}^{-1}\left (\sum_{\ell=1}^{M}s_{i}f_{\ell}s_{j}^{-1}\right )s_{j}\right )\gamma(s_{j})^{-1}\\
&=\gamma(s_{i})\sigma'\left (\sum_{\ell=1}^{M}s_{i}^{-1}s_{i}f_{\ell}s_{j}^{-1}s_{j}\right )\gamma(s_{j})^{-1}
\end{align*}
so by \cite[Lemma 3.5]{DKP1}
\begin{align*}
\left|\sigma_\gamma(s_{i}fs_{j}) -\sum_{\ell=1}^{M}\sigma_\gamma(s_{i}f_{\ell}s_{j})\right|&=\left|\sigma_\gamma(s_{i}fs_{j}) - \sum_{\ell=1}^{M}\gamma(s_{i})\sigma'(s_{i}^{-1}s_{i}f_{\ell}s_{j}^{-1}s_{j})\gamma(s_{j})^{-1}\right|\\
&<150(2|F\cup S|+1)^{2(n+5)}\delta.
\end{align*}
Thus if $f=\sum_{\ell=1}^{M}f_{\ell}$, then
\begin{align*}
\left |\sigma_\gamma(f)-\sum_{\ell=1}^{M}\sigma_\gamma(f_{\ell})\right|&=\left |\sigma_\gamma(\sum_{i,j=1}^{N}s_{i}fs_{j})-\sum_{\ell=1}^{M}\sigma_\gamma(\sum_{i,j=1}^{N}s_{i}f_{\ell}s_{j}) \right |
\\ &\leq \sum_{i,j=1}^N\left|\sigma_\gamma(s_{i}fs_{j}) -\sum_{\ell=1}^{M}\sigma_\gamma(s_{i}f_{\ell}s_{j})\right|
\\ &<150N^2(2|F\cup S|+1)^{2(n+5)}\delta.
\end{align*}

Thus to show that $\sigma_\gamma$ is $\delta'$-multiplicative it will suffice to show that for
$a,b\in(F\cup S)_\pm^{n}$ such that $ab\in (F\cup S)_\pm^{n}$ we have 
\[
|\sigma_\gamma(ab)-\sigma_\gamma(a)\sigma_\gamma(b)|<5\delta.
\]
Indeed then for $\sum_i {a_i} ,\sum_i {b_i} \in \bsigma (F \cup S)_\pm^{n}$ with $\left(\sum_i a_i\right)\left(\sum_j b_j\right) \in\bsigma (F\cup S)_\pm^{n} $
\begin{align*}
\left |\sigma_\gamma\left(\left(\sum_i a_i\right)\left(\sum_j b_j\right)\right)-\sigma_\gamma\left(\sum_i a_i\right)\sigma_\gamma\left(\sum_j b_j\right)\right|&\\
&\hskip-6cm\leq \left |\sigma_\gamma\left(\sum_{i,j} a_ib_j\right)-\left(\sum_i \sigma_\gamma(a_i)\right)\left(\sum_j \sigma_\gamma(b_j)\right)\right|+\left(150N^2(2|F\cup S|+1)^{2(2n+5)}\delta\right)^2\\
&\hskip-6cm \leq \left |\sum_{i,j} \sigma_\gamma(a_ib_j)-\sum_{i,j} \sigma_\gamma(a_i)\sigma_\gamma(b_j)\right|+2\left(150N^2(2|F\cup S|+1)^{2(2n+5)}\right)^2\delta\\
&\hskip2.5cm\text{ (assuming $\delta<1$)}\\
&\hskip-6cm \leq 5N^2\delta+2\left(150N^2(2|F\cup S|+1)^{2(2n+5)}\right)^2\delta\\
&\hskip-6cm = \delta'\\
\end{align*}

So say we have $a=s_{i_{1}}f_{1}s_{j_{1}}^{-1}$ and $b=s_{i_{2}}f_{2}s_{j_{2}}^{-1}$
where $f_{1},f_{2}\in F_\pm^{n}$. If $j_{1}\neq i_{2}$
then $s_{j_{1}}^{-1}s_{i_{2}}=\gamma(s_{j_{1}})^{-1}\gamma(s_{i_{2}})=0$
so we are done. Suppose otherwise, then 
\begin{align*}
  |\sigma_\gamma(ab)-&\sigma_\gamma(a)\sigma_\gamma(b)|
  =| \gamma(s_{i_{1}})\sigma'(s_{i_{1}}^{-1}s_{i_{1}}f_{1}(s_{j_{1}}^{-1}s_{j_{1}})^3f_{2}s_{j_{2}}^{-1}s_{j_{2}})\gamma(s_{j_{2}})^{-1}-\\&\gamma(s_{i_{1}})\sigma'(s_{i_{1}}^{-1}s_{i_{1}}f_{1}s_{j_{1}}^{-1}s_{j_{1}})\gamma(s_{j_{1}})^{-1}\gamma(s_{j_{1}})\sigma'(s_{j_{1}}^{-1}s_{j_{1}}f_{2}s_{j_{2}}^{-1}s_{j_{2}})\gamma(s_{j_{2}})^{-1}|
\end{align*}
Note first that for any $i>0$ 
\begin{align*}
|\gamma(s_{i})^{-1}\gamma(s_{i})-\sigma'(s_{i}^{-1}s_{i})| & \leq \frac{1}{d'}|\Fix(\sigma(s_{i}^{-1}s_{i})) \bigtriangleup B_{i}|+|\sigma(s_{i}^{-1}s_{i})- p_{\Fix(\sigma(s_{i}^{-1}s_{i}))}|\\
&\leq 2\delta
\end{align*}
and for $g_{1},g_{2}\in F_\pm^{2n+5}$ such that $g_{1}g_{2}\in F_\pm^{2n+5}$
\begin{align*}
| \sigma'(g_{1}g_{2})-\sigma'(g_{1})\sigma'(g_{2})| & =\frac{1}{d'}d|\sigma(g_{1}g_{2})-\sigma(g_{1})\sigma(g_{2})|\leq\delta
\end{align*}
so by repeated applications of the triangle inequality
\[
|\sigma_\gamma(ab)-\sigma_\gamma(a)\sigma_\gamma(b)|\leq2\delta+3\delta=5\delta.
\]

Let us now show that $\sigma_\gamma$ is $\delta'$-trace-preserving. Let $s_{i}fs_{j}^{-1}\in(F\cup S)_\pm^{n}$.
If $i\neq j$ then 
\[
\tr(s_{i}fs_{j}^{-1})=\tr(\gamma(s_{i})\sigma'(s_{i}^{-1}s_{i}fs_{j}^{-1}s_{j})\gamma(s_{j})^{-1})=0,
\]
suppose otherwise, then
\begin{align*}
\left|\tr(\sigma_\gamma(s_{i}fs_{i}^{-1}))-\tau(s_{i}fs_{i}^{-1})\right| & =\left|\tr\left(\gamma(s_{i})\sigma'(s_{i}^{-1}s_{i}fs_{i}^{-1}s_{i})\gamma(s_{i})^{-1}\right)-\tau(s_{i}fs_{i}^{-1})\right|\\
 & =\left|\tr\left(\sigma'(s_{i}^{-1}s_{i}fs_{i}^{-1}s_{i})\gamma(s_{i})^{-1}\gamma(s_{i})\right)-\tau(fs_{i}^{-1}s_{i})\right|\\
 & \leq\left|\tr\left(\sigma'(s_{i}^{-1}s_{i}fs_{i}^{-1}s_{i})\gamma(s_{i})^{-1}\gamma(s_{i})\right))-\tr\left(\sigma'(s_{i}^{-1}s_{i}fs_{i}^{-1}s_{i})\right)\right|\\
 &\hskip1cm+\left|\tr\left(\sigma'(s_{i}^{-1}s_{i}fs_{i}^{-1}s_{i})\right)-\frac{(N-k)\tau(fs_{i}^{-1}s_{i})}{N\mathfrak{h}(p)}\right|\\
 & \leq(2\delta+\delta)+\delta<\delta'
\end{align*}
where we used the computation above.

So we have proved 
\[
|\SA_{pGp}(F,4n+5,\delta,d)|_{E}\leq |\SA_{G}(F\cup S,n,\delta',d')|_{E\cup S} / \frac{\left(\frac {d'} N!\right)^k}{(\lceil\delta d\rceil+1)^{\lceil\delta d\rceil k}}
\]
So 
\begin{align*}
\frac{\log|\SA_{pGp}(F,4n+5,\delta,d)|_{E}}{d\log d} & \leq\frac{\log\left (|\SA_{G}(F\cup S,n,\delta',d')|_{E\cup S}/\frac{\left(\frac {d'} N!\right)^k}{(\lceil\delta d\rceil+1)^{\lceil\delta d\rceil k}}\right)}{d\log d}\\
 & =\frac{N}{N-k}\frac{\log|\SA_{G}(F\cup S,n,\delta',d')|_{E\cup S}}{d'\log d}\\
 &\hskip2cm-{k}\frac{\log(\frac{d'}{N})!}{d\log d}+{\lceil \delta d \rceil k}\frac{\log (\lceil\delta d+1 \rceil )}{d\log d}
\end{align*}

Now for $\e>0$ arbitrary and $d$ sufficiently large we have
$\frac{1}{\log d}\leq\frac{1}{\log d'}+\frac{\e}{\log d'}$
so the left hand side is at most
\[
\frac{N}{N-k}(1+\epsilon)\frac{\log|\SA_{G}(F\cup S,n,\delta',d')|_{E\cup S})}{d'\log d'}-{k}\frac{\log(\frac{d'}{N})!}{d\log d}+{\lceil \delta d \rceil k}\frac{\log(\lceil\delta d+1 \rceil )}{d\log d}
\]
By Stirlings approximation 
\[
\log(\frac{d'}{N})!=\frac{d'}{N}\log\frac{d'}{N}+\frac{d'}{N}+O(\log\frac{d'}{N})
\]
therefore
\[
s_{pGp,E}(F,2n+5,\delta)\leq\frac{N}{N-k}(1+\e)\limsup_{d\to\infty}s_{G,E\cup S}(F\cup S,n,\delta',d')+\frac{k}{N}+k\delta
\]
which gives
\begin{align*}
s_{pGp,E}(F,4n+5)&\leq\frac{1}{\mathfrak{h}(p)}(1+\e)s_{G,E\cup S}(F\cup S,n)+\frac k N\\
&=\frac{1}{\mathfrak{h}(p)}(1+\e)s_{G,E\cup S}(F\cup S,n)+1-\frac{1}{\mathfrak{h}(p)}
\end{align*}
From here it is clear that 
\[
s(pGp)\leq\frac{1}{\mathfrak{h}(p)}s(G)+1-\frac{1}{\mathfrak{h}(p)}.
\]
The same proof works with $\liminf_{d\to \infty}$ and $\lim_{d\to \omega}$ instead of $\limsup_{d\to\infty}$. 
\end{proof}

Next we prove the other direction in the general case:

\begin{lemma} $s(G)\leq\mathfrak{h}(p)s(pGp)+1-\mathfrak{h}(p)$ and similarly for $\underline s$ and $s_\omega$.
\end{lemma}

\begin{proof}

Let $S=\{s_{1},s_{2},\ldots,s_k\}\subset\bb G$ be such that
\[
s_{i}^{-1}s_{i}\leq p 
\text{\   and  \ }
\sum_{i=1}^{k}s_{i}s_{i}^{-1}=1-p.
\] 
and let $E,F\subset\bb{pGp}$ be finite subsets. We may assume that $p\in E\cap F$ and $s_{i}s_{i}^{-1},s_{i}^{-1}s_{i}\in E\cap F$ for all $i$.

Let $\sigma\in \SA_{G}(F\cup S,n+4,\delta,d)$ and assume that $d$ is large enough so $d':=\lfloor\h(p)d\rfloor>\delta^{-1}$. 
Let $B_{0}=\Fix\sigma(p)$ so
\[
\left |\frac{|B_{0}|}d -\h(p)\right|<\delta.
\]
Thus  we may arbitrarily extend or shrink $B_{0}$ to a subset $B$
such that either $B\subset B_{0}$ or $B\supset B_{0}$ and $|B|=d'$ where
\begin{align*}
\left| \sigma(p)-p_{B}\right| & \leq\left| \sigma(p)-p_{\Fix \sigma(p)}\right|+\left| p_{\Fix\sigma(p)}-p_{B}\right|\\
 & =\frac{1}{d}|\{x\in\{ 1\ldots d\} \mid \sigma(p)^{2}x\neq \sigma(p)\}|+|{\frac{1}{d}|B_{0}|-\h(p)|}+|{\frac{1}{d}|B|-\h(p)|}\\
 & <3\delta.
\end{align*}

Let $\delta'=\frac{20\delta}{\h(p)}$.
We will show that 
\[
\sigma':=\sigma_{|_{B}}\in \SA_{pGp}(F,n,\delta',d').
\]

Cleary $\sigma'$ is well defined, let us show it is $\delta'$-multiplicative that is for
$a,b\in \bsigma F_{\pm}^n$ with $ab\in \bsigma F_{\pm}^n$ 
that
\[
\left| {\sigma'}(ab)-{\sigma'}(a){\sigma'}(b)\right|<\delta'.
\]
So
\begin{align*}
\left| \sigma'(ab)-{\sigma'}(a){\sigma'}(b)\right| & ={\frac{1}{d'}\sum_{x\in B}\left|\{x\in B \mid {\sigma'}(ab)x\neq {\sigma'}(a){\sigma'}(b)x\}\right| }\\
 & ={\frac{d}{d'}}\left| p_{B}\sigma(ab)p_{B}-p_{B}\sigma(a)p_{B}\sigma(b)p_{B}\right|\\
 &\leq \frac{\delta}{\h(p)}(15\delta+5\delta)=\frac{20\delta^2}{\h(p)}< \delta' \text{, for $\delta<1$}
\end{align*}
We now show $\sigma'$ is $\delta'$-trace-preserving. Let $a\in F_{\pm}^n$ then
\begin{align*}
\left|\frac{|\Fix{\sigma'}(a)|}{d'}-\frac{\tau_{G}(a)}{\h(p)}\right| & =\frac{1}{\h(p)}\left|\frac{|\Fix{\sigma'}(a)|}{d'/\h(p)}-\tau_{G}(a)\right|\\
&\leq\frac{1}{\h(p)}\left|\frac{\h(p)}{d'}-\frac{1}{d}\right|+\frac{1}{\h(p)}\left|\frac{|\Fix{\sigma'}(a)|}{d}-\tau_{G}(a)\right|\\
 & \leq\frac{1}{\h(p)}\left(\left|\frac{|\Fix p_{B}\sigma(a)p_{B}|}{d}-\frac{|\Fix\sigma(a)|}{d}\right|+\left|\tr\circ \sigma (a)-\tau_{G}(a)\right|+\delta^2\right)\\
 & \leq\frac{1}{\h(p)}\left(\left| p_{B}\sigma(a)p_{B}-\sigma(a)\right|+\delta+\delta^2\right)\\
 & \leq\frac{1}{\h(p)}\left((6\delta+2\delta)+\delta+\delta^2\right)=\frac{9\delta+\delta^2}{\h(p)}<\delta' \text{, for $\delta<1$}
\end{align*}
So ${\sigma'}\in \SA_{pGp}(F,n,\delta',d')$ as
claimed. 

Next we study the map
\begin{align*}
\SA_{G}(F\cup S,n+4,\delta,d)&\to\SA_{pGp}(F,n,\delta',d')\\
\sigma&\mapsto \sigma'
\end{align*}
Note that 
\begin{align*}
\left|\frac{|\Fix\sigma(s_{i})\sigma(s_{i})^{-1}|}{d}-\mathfrak{h}(s_{i}s_{i}^{-1})\right| & \leq\left|\frac{|\Fix\sigma(s_{i})\sigma(s_{i})^{-1}|}{d}-\frac{|\Fix\sigma(s_{i}s_{i}^{-1})|}{d}\right|\\
&\hskip3cm +\left|{\tr\circ\sigma(s_{i}s_{i}^{-1})}-\mathfrak{h}(s_{i}s_{i}^{-1})\right|\\
 & \leq\left| \sigma(s_{i})\sigma(s_{i})^{-1}-\sigma(s_{i}s_{i}^{-1})\right|+\delta\\
 & \leq5\delta+\delta<6\delta
\end{align*}
and for $i\neq j$
\begin{align*}
|\ran(\sigma(s_{i}))\cap \ran(\sigma(s_{j}))| & =\left|\{x\in\{1,\ldots,d\}\mid \sigma(s_{j})\sigma(s_{j})^{-1}\sigma(s_{i})\sigma(s_{i})^{-1}x\neq0\}\right|\\
 & \leq d\left| \sigma(s_{j})\sigma(s_{j})^{-1}\sigma(s_{i})\sigma(s_{i})^{-1}\right|<9\delta d
\end{align*}
so
\begin{align*}
|\ran(\sigma(s_{i}))\cap B| & =\left|\{x\in\{1,\ldots,d\}\mid p_{B}\sigma(s_{i})\sigma(s_{i})^{-1}x\neq0\}\right|\\
 & \leq d\left| p_{B}\sigma(s_{i})\sigma(s_{i})^{-1}\right|\leq (3\delta+5\delta)d\leq 8\delta d
\end{align*}

So for all $d$ sufficiently large 
\[
\ran(\sigma(s_{i}))\setminus(\bigcup_{j\neq i}\ran(\sigma(s_{j}))\cup B)
\]
contains at least $d_{i}:=\lfloor{\mathfrak{h}(s_{i}s_{i}^{-1})d}-(14\delta + 9\delta k)d\rfloor$ elements so we can find 
\[
A_{i}\subset \ran(\sigma(s_{i}))\setminus(\bigcup_{j\neq i}\ran(\sigma(s_{j}))\cup B)
\]
be a subset with size $|A_{i}|=d_i$ and $A_{i}\cap A_{j}=\emptyset$
for $i\neq j$ and $A_{i}\cap B=\emptyset$.  Similarly
we have 
\[\left|\frac{|\Fix\sigma(s_{i})^{-1}\sigma(s_{i})}{d}-\mathfrak{h}(s_{i}^{-1}s_{i})\right| <6\delta \]
and
\begin{align*}
|\dom\sigma(s_{i})\bigtriangleup\Fix\sigma(s_{i}^{-1}s_{i})\cap B| & \leq|\{x\in\{1\ldots,d\}\big|P_{B}P_{\Fix\sigma(s_{i}^{-1}s_{i})}x\neq\sigma(s_{i})^{-1}\sigma(s_{i})x\}|\\
 & \leq d\left|P_{B}P_{\Fix\sigma(s_{i}^{-1}s_{i})}-\sigma(s_{i})^{-1}\sigma(s_{i})\right|<d6\delta
\end{align*}
So we can find subsets $B_{i}\subset \dom(\sigma(s_{i}))\cap \Fix\sigma(s_{i}^{-1}s_{i})\cap B$
with $|B_{i}|=d_i$. Let $d_{k+1}=d-d'-\sum_{i=1}^kd_i$. Also recall that $|\Fix \sigma(s_is_i^{-1})|\leq d(\h(s_i^{-1}s_i)+\delta)$.

Let $|\SA_{G}(F\cup S,n+4,\delta,d)\big|_{E\cup S}'$ be the number of
elements of $\SA_{G}(F\cup S,n+4,\delta,d)$ where we distinguish
elements by their values on $\{p_{B}\sigma(f)p_{B}\mid f\in E\}$ and $\{p_{A_{i}}\sigma(s_{i})p_{B_{i}}\}$.
We have shown with the above computations that 
\begin{align*}
|\SA_{G}(F\cup S,n+4,\delta,d)\big|'_{E,S}&\\
&\hskip-2cm \leq{d \choose d',d_1,\ldots,d_{k+1}}\left |\SA_{pGp}(F,n,\delta',d')\right|_{E}\\
&\hskip1cm\prod_{i=1}^{k}{d(\mathfrak{h}(s_{i}^{-1}s_{i})+\delta) \choose d_{i}}\prod_{i=1}^{k}d_{i}!
\end{align*}
where the former term accounts for the choice of the $A_i$ and $B$ and the latter terms for the choice of the $B_i$ and $\sigma(s_i)$ respectively.

Now note that for $f\in E$
\[
\left| p_{B}\sigma(f)p_{B}-\sigma(f)\right|\leq (6\delta+2\delta)=8\delta
\]
and similarly for $i\neq j$
\begin{align*}
\left|p_{A_{i}}\sigma(s_{i})p_{B_{i}}-\sigma(s_{i})\right| & \leq\left|p_{A_{i}}-\sigma(s_{i})^{-1}\sigma(s_{i})\right|+\left|p_{B_{i}}-\sigma(s_{i})\sigma(s_{i})^{-1}\right|+3\delta\\
 & \leq(14\delta+9k\delta+6\delta)+(14\delta+9k\delta+6\delta)+\delta+3\delta\text{, for d sufficiently large}\\
 & =(44+18k)\delta
\end{align*}
so we can find $\kappa>0$ with $\kappa\rightarrow0$ as $\delta\rightarrow0$
so that for a given choice of $p_{B}\sigma(f)p_{B}$ there are $d^{\kappa d}$
choices for $\sigma(f)$ and similarly for $\sigma(s_{i})$ (see \cite[Lemma 2.5]{DKP2}) so
\begin{align*}
|\SA_{G}(F\cup S,n+4,\delta,d)\big|_{E,S}&\\
&\hskip-2cm \leq{d \choose d',d_1,\ldots,d_{k+1}}\left |\SA_{pGp}(F,n,\delta',d')\right|_{E}\\
&\hskip1cm\prod_{i=1}^{k}{d(\mathfrak{h}(s_{i}^{-1}s_{i})+\delta) \choose d_{i}}\prod_{i=1}^{k}(d_{i}!)d^{\kappa d|E\cup S|}.
\end{align*}
Therefore
\begin{align*}
\log|\SA_{G}(F\cup S,n+4,\delta,d)\big|_{E,S} & \leq\log\frac {d!}{d'!d_{k+1}!}+\log\left |\SA_{pGp}(F,n,\delta',d')\right|_{E}\\
&\hskip-2cm+\sum_{i=1}^k\log \frac{d(\mathfrak{h}(s_{i}^{-1}s_{i})+\delta)!}{d_{i}!(d(\mathfrak{h}(s_{i}^{-1}s_{i})+\delta)-d_i)!}+\kappa d|E\cup S|\log d
\end{align*}
hence
\begin{align*}
\log|\SA_{G}(F\cup S,n+4,\delta,d)\big|_{E,S}+\log d'! \leq \log{d!}+\log\left |\SA_{pGp}(F,n,\delta',d')\right|_{E}+\e(\delta,d).
\end{align*}
with
\[
\e(\delta,d):=\sum_{i=1}^k\log \frac{d(\mathfrak{h}(s_{i}^{-1}s_{i})+\delta)!}{d_{i}!(d(\mathfrak{h}(s_{i}^{-1}s_{i})+\delta)-d_i)!}+\kappa d|E\cup S|\log d
\]
so since $\lim_{d\to \infty} \frac{\log d'!}{\d\log d}=\h(p)$ and $\inf_{\delta}\lim_{d\to \infty}\e(\delta,d)=0$ we obtain finally
\begin{align*}
s_{G,E\cup S}(F\cup S,n+4)+\h(p)\leq1+\inf_{\delta>0} \limsup_{d\to\infty}\frac{\left |\SA_{pGp}(F,n,\delta',d')\right|_{E}}{d\log d}.
\end{align*}
Using \cite[Lemma 2.13]{DKP2}
\begin{align*}
s_{G,E\cup S}(F\cup S,n+4)+\h(p)\leq1+\h(p)s_{pGp,E}(F,n)
\end{align*}
and we deduce
\[
s(G)\leq\mathfrak{h}(p)s(pGp)+1-\mathfrak{h}(p).
\]
The same proof works with $\liminf_{d\to \infty}$ and $\lim_{d\to \omega}$ instead of $\limsup_{d\to\infty}$. 
\end{proof}

Finally we deduce the scaling formula:

\begin{proof}[Proof of Proposition \ref{P-scaling}]
Let $p_n\leq p\leq q_n$ be projections with $\h(p_n),\h(q_n)\in \IQ$ and $|p_n-q_n|\to 0$.
\[
\h(p)s(pGp)\leq\mathfrak{h}(p_n)s(p_nGp_n)+\h(p)-\mathfrak{h}(p_n)
\]
by the previous lemma so $\liminf_{n\to\infty} s(p_nGp_n)\geq s(pGp)$. Similarly
\[
\h(q_n)s(q_nGq_n)\leq\mathfrak{h}(p)s(pGp)+\h(q_n)-\mathfrak{h}(p)
\]
by the previous lemma so $\limsup_{n\to\infty} s(q_nGq_n)\leq s(pGp)$. However the two lemmas combined imply
\[
s(G)=\mathfrak{h}(p_n)s(p_nGp_n)+1-\mathfrak{h}(p_n).
\]
so
\[
\limsup_{n\to\infty} s(p_nGp_n)=\frac 1{\h(p)}(s(g)-1)+1
\]
and
\[
s(G)=\mathfrak{h}(q_n)s(q_nGq_n)+1-\mathfrak{h}(q_n)
\]
so
\[
\liminf_{n\to\infty} s(q_nGq_n)=\frac 1{\h(p)}(s(G)-1)+1.
\]
Therfore
\[
s(pGp)=\frac 1{\h(p)}(s(G)-1)+1
\]
and the same argument works for $\underline s$ and $s_\omega$.

So
\[
\underline s(G)=s(G)\ssi \underline s(pGp)=s(pGp)
\]
so $G$ is $s$-regular if and only if $pGp$ is $s$-regular.
\end{proof}

\section{Applications of the correspondence principle and the proof of Theorem \ref{T - FPF}}\label{S- applications}

Since the work of \cite{Dye}, orbit equivalence relations have been used  to prove results in group theory on several occasions, with the Bernoulli action $\G\acts X_0^\G$ serving as a link. A recent example is the use of the Gaboriau--Lyons theorem \cite{GL} (that the orbit relation of $\G\acts [0,1]^\G$ of a non amenable group $\G$ contains a nonamenable subtreeing) as a way to replace the assumption ``contains a nonamenable free group'' on $\G$ by ``$\G$ is non amenable''.
The `correspondence principle' discussed here is a straightforward but useful extension of   this well--known  idea from groups to groupoids.

For example let us prove the following result using the principle:

\begin{proposition}\label{P - amenable groupoid actions}
Let $G\acts X$ be a pmp action of an amenable pmp groupoid $G$. Then $s(G)=s(G\ltimes X)$. In particular $s(G\ltimes X)$ doesn't depend on the action.
\end{proposition}

Let us first observe that the measure of the set of finite orbits of an action  $G\acts X$ doesn't depend on the action:

\begin{lemma}
If $G\acts X$ is an essentially free pmp action of a pmp groupoid $G$ and $D\subset X$ is a fundamental domain for the  set of finite orbits of $G$ in $X$, then 
\[
\mu(D):=\int_{G^0} \frac 1{|G_e|} \d\h(e)
\]
so $\mu(D)$ depends on $G$ only.
\end{lemma}
\begin{proof}
Let 
\[
X_\mathrm{f}:=\{x\in X\mid |Gx|<\infty\}. 
\]
Since $G\acts X$ is free $|Gx|=|G_{\r(x)}|$ almost surely so 
\[
\mu^e(X_\mathrm{f}^e)=\begin{cases} 0\text{ if }|G_e|=\infty\\ 1\text{ else}\end{cases}
\]
for ae $e\in G^0$ and if $D\subset X_{\mathrm{f}}$ is a fundamental domain then if $e_1,e_2...,e_n$ are the units isomorphic to $e$ we have
\[\sum_{i=1}^{n}\mu^{e_i} (D^{e_i})=\frac{n}{|G_e|}\]
therefore 
\[
\mu(D)=\int_{G^0}\frac 1 {|G_e|}\d\h(e).
\]
using invariance of $\mu$.
\end{proof}

\begin{remark}
An analog of the so-called ``fixed price problem'' \cite{Gab} for $s$ is the question of whether $s(G\ltimes X)$ depends  on $G$ only and not on the essentially free pmp action $G\acts X$. This ``fixed sofic dimension problem'' holds for  example for groupoids with fixed price 1 (=all their free pmp actions have cost 1) provided all their free pmp actions are sofic, but it is open in general. Also open is whether a groupoid is sofic if and only if all its free pmp actions are sofic. 
\end{remark}

\begin{proof}[Proof of Proposition \ref{P - amenable groupoid actions}] Since the Bernoulli action $G \acts [0,1]^G$ is essentially free $G\ltimes [0,1]^G$ is a pmp equivalence relation so  by Theorem \ref{T-action Bernoulli infinite} and  \cite[Corollary 5.2]{DKP1}
\[
s(G)=s(G\ltimes [0,1]^G)=1-\mu(D)=1-\int_{G^0} \frac 1{|G_e|} \d\h(e)
\]
where $D$ is the fundamental domain of the set of finite classes of the amenable pmp equivalence relation $G\ltimes [0,1]^G$.

If $G\acts X$ is a pmp action then the action $G\ltimes X\acts [0,1]^{G\ltimes X}$ is pmp and essentially free so by  Theorem \ref{T-action Bernoulli infinite}
\[
s(G\ltimes X)=s\left ((G\ltimes X)\ltimes [0,1]^{G\ltimes X}\right).
\]
Now the measure isomorphism 
\begin{align*}
(G\ltimes X)\ltimes [0,1]^{G\ltimes X}&\to G\ltimes \left (X\times [0,1]^{G\ltimes X}\right)\\
((s,x),y)&\mapsto (s,(x,y))
\end{align*}
is an isomorphism of pmp equivalence relations, where the action $G\acts X\times [0,1]^{G\ltimes X}$ is diagonal and $G\acts [0,1]^{G\ltimes X}$ on the first coordinate. 

So
\begin{align*}
s(G\ltimes X)&=s\left (G\ltimes \left (X\times [0,1]^{G\ltimes X}\right)\right)\\
&=1-\mu(D')\\
&=1-\int_{G^0} \frac 1{|G_e|} \d\h(e).
\end{align*}
where $D'$ is the is the fundamental domain of the set of finite classes of the amenable pmp equivalence relation $G\acts X\times [0,1]^{G\ltimes X}$. 
So 
\[
s(G)=s(G\ltimes X).
\]
\end{proof}

Another example is the proof that the invariance of $s$ under orbit equivalence (Theorem \ref{T - OE invariance}) established in \cite[Theorem 4.1]{DKP1} is  equivalent to the statement for groupoids  in \cite[Theorem 2.1]{DKP2}.

\begin{theorem}
Let $G$ be a pmp groupoid, $E,F$ be transversally generating sets, then $s(E)=s(F)$, $\underline s(E)=\underline s(F)$ and $s^\omega(E)=s^\omega(F)$.
\end{theorem}

\begin{proof}
If $F$ is a transversally generating set of $G$ (which we assume to have infinite fibers) then $F\cup \{p_{B_0},p_{B_1}\}$ (as defined in Lemma \ref{lemma B0 def}) is a transversally generating set of the pmp equivalence relation $G\ltimes \{0,1\}^G$. By Theorem \ref{T-action Bernoulli infinite} 
\[
s(F)=s(F\cup \{p_{B_0},p_{B_1}\})
\] 
so by Theorem \ref{T - OE invariance} we have
\[
s(E)=s(E\cup \{p_{B_0},p_{B_1}\})=s(F\cup \{p_{B_0},p_{B_1}\})=s(F).
\]
The same applies to $\underline s$ and $s_\omega$.
\end{proof}

We conclude with the proof Theorem \ref{T - FPF} which is our main illustration of the correspondence principle.

\begin{proof}[Proof of Theorem \ref{T - FPF}] Let $\tilde G_i:={G_i}_{|G_3^0}$ and $\tilde G:=G_{|G_3^0}=\tilde G_{1}*_{G_{3}}\tilde G_{2}$ then for every essentially free pmp action $\tilde G\acts X$ of $\tilde G$ we have
\begin{align*}  
s(G)-1& =\h(G_3^0)  \left ( s(\tilde G_{1}*_{G_{3}}\tilde G_{2})-1\right)\ \text{by Proposition \ref{P-scaling}}\\
& \geq \h(G_3^0)  \left (  s(\tilde G_{1}*_{G_{3}}\tilde G_{2}\ltimes X)-1\right)\ \text{by Proposition \ref{P-action general}}\\
 & =\h(G_3^0)  \left ( s\left((\tilde G_{1}\ltimes X_{|1})*_{(G_{3}\ltimes X_{|3})}(\tilde G_{2}\ltimes X_{|2})\right)-1\right)\text{ by Lemma \ref{L - free product decompositions actions}}\\
 & =\h(G_3^0)  \left (s(R_{1}*_{R_{3}}R_{1})-1\right)\text{ where } R_i:=\tilde G_{i}\ltimes X_{|i} \text{  are pmp }\\
 &\hskip1cm  \text{ equivalence relations by Lemma \ref{L - free groupoid actions are equivalence relations}}\\
 & =\h(G_3^0)  \left (s(R_{1})+s(R_{2})-s(R_3)-1\right)\text{ by \cite[Theorem 1.2]{DKP1}}\\
 & =\h(G_3^0)  \left (s(R_{1})+s(R_{2})-s(G_3)-1\right)\text{ by Proposition \ref{P - amenable groupoid actions}}
\end{align*}
If $\tilde G\acts X$ is Bernoulli then  
\[
s(\tilde G_{1}*_{G_{3}}\tilde G_{2}) = s(\tilde G_{1}*_{G_{3}}\tilde G_{2}\ltimes X)
\]
by Theorem \ref{T-action Bernoulli infinite} and $\tilde G_{i}\acts X_{|i}$ for $i=1,2$ are (isomorphic to) Bernoulli actions by Lemma \ref{L - subshifts are Bernoulli}. Therefore $s(R_i)=s(\tilde G_i)$, $i=1,2$ by Theorem \ref{T-action Bernoulli infinite}. Then
\begin{align*}  
s(G)-1& =\h(G_3^0)  \left (s(R_{1})+s(R_{2})-s(G_3)-1\right)\\
&=\h(G_3^0)  \left (s(\tilde G_1)+s(\tilde G_{2})-s(G_3)-1\right)\text{ by Lemma \ref{L - subshifts are Bernoulli} and }\\
&\hskip7cm\text{ Theorem \ref{T-action Bernoulli infinite} }\\
&=\h(G_3^0)  \left (s(\tilde G_1)-1\right)+\h(G_3^0)  \left (s(\tilde G_{2})-1\right)-\h(G_3^0)  \left (s(G_3)-1\right)\\
&=\h(G_1^0)  \left (s(G_1)-1\right)+\h(G_2^0)  \left (s(G_{2})-1\right)-\h(G_3^0)  \left (s(G_3)-1\right)\\
&\hskip7cm \text{ by Proposition \ref{P-scaling}}\\
&=\h(G_1^0)  s(G_1)+\h(G_2^0)s(G_{2})-\h(G_3^0) s(G_3) -\h(G_1^0)-\h(G_2^0)+\h(G_3^0)\\
&=\h(G_1^0)  s(G_1)+\h(G_2^0)s(G_{2})-\h(G_3^0) s(G_3) -1
\end{align*}
(We note that the inequality $s(G_{1}*_{G_{3}}G_{2})\leq s(G_{1})+s(G_{2})-s(G_3)$ could also be proved by a direct argument.)
\end{proof}

\end{document}